\documentclass[12pt]{amsart}
\usepackage[latin1]{inputenc}
\usepackage[T1]{fontenc}
\usepackage{amsmath, amssymb}
\usepackage{amsthm}
\usepackage{lastpage}
\usepackage{fancyhdr}
\usepackage{graphicx}
\usepackage{bbm}
\usepackage{empheq}
\usepackage{geometry}
\usepackage[matrix, arrow, curve]{xy}
\usepackage{color}
\definecolor{darkblue}{rgb}{0,0,0.6}
\usepackage[ocgcolorlinks,colorlinks=true, citecolor=darkblue, filecolor=darkblue,
linkcolor=darkblue, urlcolor=darkblue]{hyperref}
\usepackage[nameinlink,capitalise,noabbrev]{cleveref}
\numberwithin{equation}{section}
\DeclareMathOperator{\asdim}{asdim}

\DeclareMathOperator*{\supp}{supp}

\DeclareMathOperator{\diam}{diam}
\DeclareMathOperator{\pr}{pr}
\DeclareMathOperator{\Map}{Map}

\newcommand{\mcn}{\mathfrak{n}}
\newcommand{\mcp}{\mathfrak{p}}

\newcommand{\bbF}{\mathbbm{F}}
\newcommand{\bbZ}{\mathbbm{Z}}
\newcommand{\bbK}{\mathbbm{K}}

\newcommand{\bbN}{\mathbbm{N}}
\newcommand{\bbL}{\mathbbm{L}}

\newcommand{\mcC}{\mathcal{C}}

\newcommand{\mcD}{\mathcal{D}}

\newcommand{\mcS}{\mathcal{S}}
\newcommand{\mcB}{\mathcal{B}}
\newcommand{\mcA}{\mathcal{A}}

\newcommand{\mcU}{\mathcal{U}}
\newcommand{\mcF}{\mathcal{F}}

\newcommand{\mcT}{\mathcal{T}}

\DeclareMathOperator{\im}{im}

\renewcommand{\phi}{\varphi}
\DeclareMathOperator*{\colim}{colim}
\numberwithin{equation}{section}
\newtheorem{thm}[equation]{Theorem}
\newtheorem{prop}[equation]{Proposition}
\newtheorem{cor}[equation]{Corollary}
\newtheorem{lemma}[equation]{Lemma}

\newtheorem*{thm*}{Theorem}
\newtheorem{thmA}{Theorem}

\newtheorem*{cor*}{Corollary}
\newtheorem*{prop*}{Proposition}
\newtheorem*{lemma*}{Lemma}
\newtheorem*{mconj*}{Meta-Conjecture}
\newtheorem*{conj*}{Conjecture}
\newtheorem{cor+}{Corollary}
\theoremstyle{definition}

\newtheorem{defi}[equation]{Definition}
\newtheorem{rem}[equation]{Remark}
\newtheorem{nota}[equation]{Notation}

\newtheorem*{bsp*}{Example}
\newtheorem*{example*}{Example}
\newtheorem*{defi*}{Definition}
\newtheorem*{rem*}{Remark}
\newtheorem*{nota*}{Notation}
\begin{document}
\title{On the K-theory of groups with Finite Decomposition Complexity}
\author{Daniel Kasprowski}
\date{\today}
\subjclass[2010]{19D50 (Primary) 19G24, 20F69 (Secondary)}

\address{Fachbereich Mathematik und Informatik,
 Westf\"{a}lische Wilhelms-Universit\"{a}t M\"{u}nster,
 Einsteinstr. 62, D-48149 M\"{u}nster, Germany}
\email{daniel.kasprowski@math.uni-muenster.de}
%\subjclass[2010]{18F25, 19D50 (Primary) 19G24, 20F69 (Secondary)}
%\keywords{}

\thanks{This work was supported by the SFB 878
 ``Groups, Geometry and Actions''.}

\begin{abstract}
It is proved that the assembly map in algebraic $K$- and $L$-theory with respect to the family of finite subgroups is injective for groups $\Gamma$ with finite quotient finite decomposition complexity (a strengthening of finite decomposition complexity introduced by Guentner, Tessera and Yu) that admit a finite dimensional model for $\underbar E\Gamma$ and have an upper bound on the order of their finite subgroups. In particular, this applies to finitely generated linear groups over fields with characteristic zero with a finite dimensional model for $\underbar E\Gamma$.
\end{abstract}
\maketitle
\section*{Introduction}
\noindent Assembly maps in algebraic $K$- and $L$-theory are a useful tool to study the $K$-~and $L$-theory of group rings $R[\Gamma]$. This is important for understanding geometric properties of manifolds with fundamental group $\Gamma$.

More precisely, for every ring $R$ and any group $\Gamma$ we can consider the functor $\bbK_R\colon Or\Gamma\to \mathfrak{Spectra}$ from the orbit category of $\Gamma$ to the category of spectra sending $\Gamma/H$ to (a spectrum weakly equivalent to) the $K$-theory spectrum $\bbK(R[H])$ for every subgroup $H\leq\Gamma$. For any such functor $F\colon Or\Gamma\to \mathfrak{Spectra}$ a $\Gamma$-homology theory $\bbF$ can be constructed via
\[\bbF(X):=\Map_\Gamma(\_,X_+)\wedge_{Or\Gamma}F,\]
see \cite{davislueck}. We will denote its homotopy groups by $H_n^\Gamma(\_,F):=\pi_n\bbF(X)$. The assembly map for the family of finite subgroups is the map
\begin{equation}
H_n^\Gamma(\underbar E\Gamma;\bbK_R)\to H^\Gamma_n(pt;\bbK_R)\cong K_n(R[\Gamma])\label{eq1}
\end{equation}
induced by the map $\underbar E\Gamma\to pt$.\\

Finite decomposition complexity, introduced by Guentner, Tessera and Yu in \cite{fdc} and \cite{rigidity}, is a generalization of finite asymptotic dimension. In \cite{k-theory} Ramras, Tessera and Yu prove the injectivity of the assembly map for algebraic $K$-theory for groups $\Gamma$ with finite decomposition complexity and finite classifying space $B\Gamma$. In this article we use the methods of Bartels and Rosenthal \cite{bartels} to relax the finiteness assumption on $B\Gamma$ to allow for groups with torsion. To do so we need the stronger notion of finite quotient finite decomposition complexity. Our main result is the following theorem.
\begin{thmA}
\label{thmA}
Let $\Gamma$ be a discrete group with finite quotient finite decomposition complexity and a global upper bound on the order of its finite subgroups and let $R$ be a ring. Assume there exists a finite dimensional model for $\underbar E\Gamma$. Then the assembly map 
\[H_*^\Gamma(\underbar E\Gamma;\bbK_R)\rightarrow K_*(R[\Gamma])\]
in algebraic $K$-theory is a split injection.
\end{thmA}
In \cref{proof} \cref{main} is proved. This is a slightly stronger version than \cref{thmA}. In \cref{main} the assembly map is proved to be injective for every small additive $\Gamma$-category $\mcA$. \cref{thmA} follows from this by taking $\mcA$ to be (a small skeleton of) the category of finitely generated free $R$-modules with trivial $\Gamma$-action.\\
In \cref{ltheory} an $L$-theoretic version of \cref{main} is proved.\\
~\\
If $\Gamma$ is torsion-free, then having finite quotient finite decomposition complexity is the same as having finite decomposition complexity. So we immediately get the following.
\begin{cor+}
Let $\Gamma$ be a discrete group with finite decomposition complexity that admits a finite dimensional model for $E\Gamma$ (in particular $\Gamma$ is torsion-free). Then the assembly map \eqref{eq1} is split injective for every ring $R$.
\end{cor+}
Since groups with finite asymptotic dimension have finite quotient finite decomposition complexity, we also get the following strengthening of Theorem A of \cite{bartels}.
\begin{cor+}
Let $\Gamma$ be a discrete group with finite asymptotic dimension and a global upper bound on the order of its finite subgroups. Assume there exists a finite dimensional model for $\underbar E\Gamma$. Then the assembly map \eqref{eq1} is split injective for every ring $R$.
\end{cor+}
In \cite{bartels} no assumption on the order of the finite subgroups is made but there is a gap in the proof of Proposition 7.4., namely wedge products and homotopy fixed points are interchanged, at this point Lemma 7.1(i) can not be applied as stated. Therefore, the proofs of the main results from \cite{bartels} are only correct under the additional assumption that there exists a cocompact model for $\underbar E\Gamma$. In \cref{des} we prove a bounded version of the Descent Principle. This can be used to fix the proofs of the main theorems of \cite{bartels} under the additional assumption of an upper bound on the order of finite subgroups. If the proof of the Descent Principle 7.5 in \cite{bartels} can be fixed without further assumptions, then our main theorem holds without the assumption on the order of the finite subgroups.\\
\\
In \cite{fdc} it is shown that linear groups have finite decomposition complexity and a slight modification of this proof will yield that they even have finite quotient finite decomposition complexity, as we will show in \cref{groups}.  By Selberg's Lemma \cite{selberg} a finitely generated linear group over a field of characteristic zero is virtually torsion-free and hence contains a torsion-free normal subgroup of finite index. Thus, it has an upper bound on its finite subgroups. By Alperin and Shalen \cite{alperin} it was proved that a finitely generated subgroup of a linear group over a field of characteristic zero has finite virtual cohomological dimension if and only if there is a global upper bound on the Hirsch rank of its unipotent subgroups. Furthermore, if a group has finite virtual cohomological dimension it admits a finite dimensional model for $\underbar E\Gamma$ by \mbox{\cite[Theorem 6.4]{MR1757730}}. Therefore, we get the following corollary:
\begin{cor+}
Let $F$ be a field of characteristic zero, $\Gamma$ a finitely generated subgroup of $GL_n(F)$ with a global upper bound on the Hirsch rank of its unipotent subgroups. Then the $K$-theoretic assembly map \eqref{eq1} is split injective for every ring $R$.
\end{cor+}
A finitely generated linear group $\Gamma$ over a field of positive characteristic has finite asymptotic dimension by \cite[Theorem 3.1]{rigidity} and $\underbar E\Gamma$ admits a finite dimensional model by \cite[Corollary 5]{poschar}.
\begin{cor+}
Let $F$ be a field of positive characteristic, $\Gamma$ a finitely generated subgroup of $GL_n(F)$. Suppose $\Gamma$ has an upper bound on the order of its finite subgroups, then the $K$-theoretic assembly map \eqref{eq1} is split injective for every ring $R$.
\end{cor+}
\textbf{Acknowledgments:} This article is part of my PhD thesis. I would like to thank Fabian Hebestreit, Christoph Winges and especially my supervisor Arthur Bartels for many helpful discussions. I would also like to thank the referee for his suggestions improving this article.
\tableofcontents
~\\
\begin{nota*}~
\begin{itemize}
\item $\Gamma$ will always denote a group, and all groups will be countable and discrete.
\item $\mcA$ will always denote an additive category, and all categories will be small.
\item Metrics are allowed to take on the value $\infty$ and a \emph{metric $\Gamma$-space} will always be a metric space with an isometric (left) $\Gamma$-action. %Metrics on groups are always finite.
\item For a metric space $X$, a subspace $Y\subseteq X$ and $R>0$ define
\[Y^R:=\{x\in X\mid d(x,Y)<R\}.\]
\item For metric spaces $\{(X_i,d_i)\}_{i\in I}$ we define $\coprod_{i\in I}X_i$ to be the formal set-theoretic disjoint union of the spaces $X_i$, i.e. points in $\coprod_{i\in I}X_i$ are pairs $(x,i)$ with $x\in X_i$ and we give this space the metric
\[d(x,i),(y,j))=\left\{\begin{matrix} d_{i}(x,y)&\text{if }i=j\\\infty&\text{else}\end{matrix}\right.\]
%\item Most spaces in this article are locally compact. In particular, if $A$ is \mbox{$\Gamma$-invariant} subset of a locally compact $\Gamma$-space $X$, then  $A$ is cocompact if and only if there exists a compact subset $K\subseteq X$ with $\Gamma K=A$. For this reason when considering cocompact subsets we restrict to subsets of the form $\Gamma K$ with $K\subseteq X$ compact.
\end{itemize}
\end{nota*}
%\section{Metric properties of \underbar E$\Gamma$}
\section{Metric properties of \texorpdfstring{\underline E$\Gamma$}{EG}}
In this section we will first show that any discrete, countable group $\Gamma$ that admits a finite dimensional model for $\underbar E\Gamma$ also admits a finite dimensional simplicial model with a proper $\Gamma$-invariant metric.
\begin{defi}
A metric space $(X,d)$ (resp. the metric $d$) is called \emph{proper} if for every $R>0$ and every $x\in X$ the closed ball $\bar B_R(x):=\{y\in X\mid d(x,y)\leq R\}\subseteq X$ is compact.\\
A metric $d$ on $X$ is called \emph{finite} if $d(x,y)<\infty$ for all $x,y\in X$.
\end{defi}
\begin{lemma}
\label{metric1}
Let $X$ be a finite dimensional $\Gamma$-CW-complex with countably many cells. Then $X$ is $\Gamma$-homotopy equivalent to a (countable) simplicial $\Gamma$-CW-complex of the same dimension.
\end{lemma}
This lemma is stated in \cite{mislin} and is proved in \cite[2C.5]{hatcher} for $\Gamma=\{e\}$. The construction from this proof can be done in an equivariant fashion.
\begin{lemma}
\label{metric2}
Let $X$ be a finite dimensional, countable (simplicial) $\Gamma$-CW-complex with finite stabilizers. Then $X$ is $\Gamma$-homotopy equivalent to a locally finite, finite dimensional, countable (simplicial) $\Gamma$-CW-complex.
\end{lemma}
\begin{proof}
Let $\{\sigma_n\}_{n\in\bbN}$ be an enumeration of the $\Gamma$-cells of $X$. For $t\geq 0$ let $X_t$ be the smallest subcomplex of $X$ containing $Y_t:=\{\sigma_n\mid n\leq \lfloor t\rfloor\}$. Since $Y_t$ contains only finitely many $\Gamma$-cells, $X_t$ is finite for every $t\geq 0$ as well.\\
The mapping telescope $T:=\{(x,t)\in X\times [0,\infty)\mid x\in X_t\}$ is a $\Gamma$-CW complex and since $X$ has only finite stabilizers, $T$ is locally finite.\\
The natural projection $p:T\rightarrow X$ is bijective on $\pi_0$. Let $x_0\in X$ and choose $t>0$ with $x_0\in X_t$. Let $f:(S^n,pt)\rightarrow (X,x_0)$ be a pointed map. There exists $t'$ with $f(S^n)\subseteq X_{t'}$. The inclusion $X_{t'}\times\{t'\}\subseteq T$ gives a map \mbox{$g':(S^n,pt)\rightarrow (T,(x_0,t'))$} with $p\circ g'=f$. Using the linear path from $(x_0,t)$ to $(x_0,t')$ we can construct a map \mbox{$g:(S^n,pt)\rightarrow (T,(x_0,t))$} with $p\circ g\simeq f$. Therefore, $p$ is surjective on $\pi_n$. A similar argument shows that $p$ is injective on $\pi_n$ as well.\\
By the same argument for each subgroup $G\leq \Gamma$ the projection of the fixed point spaces $p:T^G\rightarrow X^G$ is an isomorphism for all homotopy groups.\\
Since both $T$ and $X$ are $\Gamma$-CW complexes, the map $p$ is a $\Gamma$-homotopy equivalence by \cite[Propostion II.2.7]{tomdieck}.\\
If $X$ is simplicial, then there is a simplicial structure on $T$ with vertices $(v,n)$, where $n\in\bbN$ and $v$ a vertex of $X_n$. This comes from triangulating each prism $\sigma\times[n,n-1]$, where $\sigma$ is a simplex of $X_n$. 
\end{proof}
\begin{lemma}
\label{countable}
If $X$ is a model for $\underbar E\Gamma$, then $X$ contains a countable $\Gamma$-subcomplex which is still a model for $\underbar E\Gamma$.
\end{lemma}
\begin{proof}
Because $\Gamma$ has only countably many finite subgroups there exists a countable $\Gamma$-subcomplex $X_0\subseteq X$ with fixed point sets $X_0^G\neq \emptyset$ for all $G\leq\Gamma$ finite.\\
Inductively define countable $\Gamma$-subcomplexes $X_i\subseteq X$ such that $X_{i-1}^G\hookrightarrow X_i^G$ is null homotopic for every finite subgroup $G\leq \Gamma$. Those exist because $X^G$ is contractible and they can be chosen as countable complexes since the image of every contraction of $X_{i-1}^G$ in $X^G$ lies in a countable subcomplex.\\
Since $(\bigcup_{i\in\bbN} X_i)^G=\bigcup_{i\in \bbN} X_i^G$ and $X_i^G$ is contractible in $X_{i+1}^G$, the subcomplex $(\bigcup_{i\in \bbN}X_i)^G$ has vanishing homotopy groups and is therefore contractible. So $\bigcup_{i\in\bbN}X_i$ is a countable subcomplex of $X$ which is still a model for $\underbar E\Gamma$.
\end{proof}
By combining the three lemmas above we get:
\begin{prop}
\label{metric}
If $\Gamma$ admits a finite dimensional model for $\underbar E\Gamma$, then it also admits a finite dimensional simplicial $\Gamma$-CW-model for $\underbar E\Gamma$ with a proper $\Gamma$-invariant metric.
\end{prop}
\begin{proof}
By \cref{countable} there exists a countable, finite dimensional $\Gamma$-CW-model for $\underbar E\Gamma$. So by \cref{metric1} and \cref{metric2} the group $\Gamma$ admits a locally finite, finite dimensional, simplicial $\Gamma$-CW-model $X$ for $\underbar E\Gamma$. We can take the unique path metric on $X$ that restricts to the Euclidean metric on every simplex. For a definition of the simplicial path metric see \cite[Section 2]{k-theory}. This metric has the desired properties.
\end{proof}
\subsection*{Uniform contractibility}
In the following we will need some contractibility properties. We will first define uniform contractibility with respect to subspaces and then prove the results about uniform contractibility from \cite{bartels} in the relative case.
\begin{rem}
Let $G$ be a finite group, $X$ a metric $G$-space and $q\colon X\rightarrow G\backslash X$ the quotient map. Then
\[d_{G\backslash X}(y,y'):=d_X(q^{-1}(y),q^{-1}(y'))\quad\text{for}\;y,y'\in G\backslash X\]
defines a metric on $G\backslash X$. We will always consider quotients of metric spaces by an isometric action of a finite group as metric spaces using this metric.
\end{rem}
In \cite[Definition 1.1]{bartels} uniform $\mcF in$-contractibility was defined. We will need a relative version of this definition.
\begin{defi}
Let $X$ be a proper metric $\Gamma$-space, $Y\subseteq X$. $X$ is said to be \emph{uniformly $\mcF in$-contractible with respect to $Y$} if for every $R>0$ there exists an $S>0$ such that the following holds: For every $G\leq\Gamma$ finite and every $G$-invariant subset $B\subseteq X$ of diameter less than $R$ with $B\cap Y^R\neq\emptyset$ the inclusion $B\cap Y^R\hookrightarrow B^S$ is $G$-equivariantly null homotopic.\\
We say that a metric space $X$ is \emph{uniformly contractible with respect to $Y\subseteq X$} if for every $R>0$ there is an $S>0$ such that $B_R(x)\cap Y^R\hookrightarrow B_S(x)$ is null homotopic for every $x\in X$ with $B_R(x)\cap Y^R\neq\emptyset$. \\
Note that every metric space is uniformly contractible with respect to the empty subset.
\end{defi}
In contrast to the definition in \cite{bartels} we require $S$ to be independent of the finite subgroup $G$.\\
In \cite[Lemma 1.5]{bartels} it was proved that if a group $\Gamma$ admits a cocompact model $X$ for $\underbar E\Gamma$, then $X$ is uniformly $\mcF in$-contractible. This can be modified to show the following relative version.
\begin{lemma}
\label{uni1}
Let $X$ be a $\Gamma$-CW-model for $\underbar E\Gamma$ with a proper $\Gamma$-invariant metric and let $K\subseteq X$ be a finite subcomplex. Then $X$ is uniformly $\mcF in$-contractible with respect to $\Gamma K$.
\end{lemma}
\begin{proof}
Let $R>0$ and a finite subcomplex $K\subseteq X$ be given. Because the metric on $X$ is proper $K^R$ is contained in a finite subcomplex $K'\subseteq X$. Let $k$ be an upper bound on the diameter of cells in $\Gamma K'$. Every subspace $Y\subseteq \Gamma K'$ (invariant under some $G\leq\Gamma$) of diameter less than $R$ is contained in a finite subcomplex of $\Gamma K'$ (invariant under $G$) of diameter less than $R':=R+k$. Let $\mcB$ be the set of all finite subcomplexes of $\Gamma K'$ of diameter less than $R'$. Define $\mcT:=\{(B,G)\mid B\in\mcB,G\leq\Gamma_B\}$, where $\Gamma_B:=\{\gamma\in\Gamma\mid\gamma B=B\}$. $\Gamma$ acts on $\mcT$ by $\gamma(B,G):=(\gamma B,G^{\gamma})$, where $G^{\gamma}:=\gamma G\gamma^{-1}$. Since the metric on $X$ is proper there exists a finite subcomplex $K''$ containing $K'^{R'}$. For each $(B,G)\in\mcT$ there exists $\gamma\in \Gamma$ with $\gamma B\cap K'\neq\emptyset$ and therefore $\gamma B\in K''$. The complex $K''$ is finite and so up to translation by $\Gamma$ there are only finitely many choices of $B$. Furthermore, the subgroups $\Gamma_B$ are finite since the $\Gamma$-action on $X$ is proper. So for each $(B,G)\in\mcT$ with fixed $B$ there are only finitely many choices of $G$. This implies that the quotient $\Gamma\backslash\mcT$ is finite.\\
For every $(B,G)$ there exists an $S=S(B,G)>0$ such that $B$ is $G$-equivariantly contractible in $B^S$ since $X$ is $G$-equivariantly contractible by assumption. Because the action of $\Gamma$ on $\mcT$ has finite quotient, this $S$ can be chosen independently of $B$ and $G$.
\end{proof}
\begin{nota}
\label{nota:1.9}
Let $X$ be a space with an action of a group $\Gamma$ and let $\mcS$ be a collection of subsets of $\Gamma$. Define $X^\mcS$ to be the union of all fixed sets $X^H$, where $H\in\mcS$. If $\mcS$ is closed under conjugation by elements of $\Gamma$, then $X^\mcS$ is a $\Gamma$-invariant subspace of~$X$.
\end{nota}
Also \cite[Lemma 1.4]{bartels} can be modified to a relative version.
\begin{lemma}
\label{uni2}
Let $X$ be a $\Gamma$-space with a proper $\Gamma$-invariant metric and $Y\subseteq X$ be a $\Gamma$-invariant subspace. For every finite subgroup $G\leq\Gamma$ let $J(G)$ be the set of families $\mcS$ of subgroups of $G$ such that $\mcS$ is closed under conjugation by $G$. For all $n\in\bbN$ define
\[J_n:=\{(G,\mcS)\mid G\leq\Gamma, |G|\leq n, \mcS\in J(G)\}\]
and assume $X$ is uniformly $\mcF$in-contractible with respect to $Y$. Then $\coprod_{(G,\mcS)\in J_n}G\backslash X^\mcS$ is uniformly contractible with respect to $\coprod_{(G,\mcS)\in J_n}G\backslash(Y\cap X^\mcS)$ for every $n\in\bbN$.
\end{lemma}
\begin{proof}
Let $n\in\bbN$ and $R>0$ be given. Let $S>0$ be such that for every finite subgroup $G\leq\Gamma$ and every $G$-invariant subset $B\subseteq X$ of diameter less or equal to $2Rn$ with $B\cap Y^R\neq\emptyset$ the inclusion $B\cap Y^R\hookrightarrow B^S$ is $G$-equivariantly null homotopic.\\
Let $(G,\mcS)\in J_n$, $y\in G\backslash X^\mcS$ and $x\in q^{-1}(y)$, where $q\colon X^\mcS\rightarrow G\backslash X^\mcS$ is the quotient map. Let $H$ be the subgroup of $G$ consisting of all $g\in G$ for which there is a sequence $g_1,..,g_m\in G$ such that $g_1=e$, $g_m=g$ and $d(g_ix,g_{i+1}x)\leq 2R$. This is indeed a subgroup since for $g,h\in H$ with sequences $e=g_1,..,g_{m}=g$ and $e=h_1,..,h_{m'}=h$ as above the sequence $e=g_1,..,g_m=g,gh_1,..,gh_{m'}=gh$ shows that $gh\in H$. The diameter of $B:=Hx^R\subseteq X$ is bounded by $2R|G|\leq 2Rn$. Therefore, if $B\cap Y^R\neq\emptyset$, the inclusion $B\cap Y^R\hookrightarrow B^S$ is $H$-equivariantly null homotopic. In particular let $z\in B^S$ be the point fixed by $H$ on which $B\cap Y^R$ contracts (for some choice of a null homotopy). 
Let $g\in G$ and $y\in gB\cap B$, then there are $h_1,h_2\in H,x_1,x_2\in x^R$ such that $y=gh_1x_1=h_2x_2$. Then
\[d(h_2^{-1}gh_1x,x)\leq d(h_2^{-1}gh_1x,h_2^{-1}gh_1x_1)+d(h_2^{-1}gh_1x_1,x_2)+d(x_2,x)\leq R+0+R\]
and $h_2^{-1}gh_1\in H$ and thus also $g\in H$. This shows that for $g\in G-H$ we have $gB\cap B=\emptyset$ and so for $g,g'\in G$ the sets $gB$ and $g'B$ only intersect non-trivially if they are equal. Therefore, the inclusion $GB\cap Y^R\hookrightarrow GB^S$ is $G$-equivariantly homotopic to a map that sends $gB$ to $gz$. By $G$-equivariance this homotopy can be restricted to $X^\mcS$, which induces the required null homotopy on the quotient.
\end{proof}
\section{Finite quotient Finite Decomposition Complexity}
\begin{defi}
A metric space $X$ decomposes over a class of metric spaces $\mcC$ if for every $r>0$ there exists a decomposition $X=U_r\cup V_r$ and further decompositions
\[U_r=\bigcup_{i\in I_r}U_{r,i},\quad V_r=\bigcup_{j\in J_r}V_{r,j}\]
such that $d(U_{r,i},U_{r,i'})>r$ and $d(V_{r,j},V_{r,j'})>r$ for all $i,i'\in I_r,j,j'\in J_r$ with $i\neq i',j\neq j'$ and such that the metric spaces $\coprod_{i\in I_r}U_{r,i}$ and $\coprod_{j\in J_r}V_{r,j}$ are in~$\mcC$.\\
A class of metric spaces $\mcC$ is stable under decomposition if every metric space which decomposes over $\mcC$ actually belongs to $\mcC$.
\end{defi}
A metric space $X$ is called \emph{semi-bounded} if there exists $R>0$ such that for all $x,y\in X$ either $d(x,y)<R$ or $d(x,y)=\infty$.
\begin{defi}
The class of metric spaces $\mcD$ with \emph{finite decomposition complexity} is the minimal class of metric spaces which contains all semi-bounded spaces and is stable under decomposition. We say $X$ has FDC if it is contained in $\mcD$.
\end{defi}
This definition of FDC is the same as the original definition (\cite[Definition 2.1.3]{fdc}) but using only metric spaces instead of metric families. A metric family $\{X_i\}$ has FDC in the sense of \cite{fdc} if and only if $\coprod X_i$ has FDC as in the above definition.
\begin{defi}
Let $X$ be a metric $\Gamma$-space and $Y\subseteq X$ a subspace. We say that $Y$ has \emph{finite quotient FDC} (fqFDC) if for every $n\in \bbN$ the space \[\coprod_{G\leq\Gamma, |G|\leq n}G\backslash GY\] has FDC.
\end{defi}
\begin{rem}
As remarked in the introduction if \cite[Theorem 7.5]{bartels} is true, then we do not need the assumption on the upper bound on the order of the finite subgroups of $\Gamma$ in our main result, \cref{main}. In that case it would also suffice that $G\backslash \Gamma$ has FDC for every finite subgroup of $\Gamma$ (and a proper invariant metric on $\Gamma$) instead of assuming that $\coprod_{G\leq \Gamma, |G|\leq n}G\backslash \Gamma$ has FDC for every $n\in\bbN$.
\end{rem}
Recall the following definition of asymptotic dimension. By \cite[Theorem 9.9]{roe} it is equivalent to the original definition of Gromov \cite{gromov}.
\begin{defi}
Let $X$ be a metric space. The \emph{dimension of a cover} $\mcU$ of $X$ is the smallest integer $n$ such that each $x\in X$ is contained in at most $n+1$ members of $\mcU$. The cover $\mcU$ is called \emph{bounded} if $\sup_{U\in\mcU}\diam U<\infty$. $\mcU$ has \emph{Lebesgue number} at least $R>0$ if the ball of radius $R$ around $x$ is contained in some $U\in\mcU$ for each $x\in X$.\\
The \emph{asymptotic dimension} of $X$ is the smallest integer $n$ such that for any $R>0$ there exists an $n$-dimensional bounded cover $\mcU$ of $X$ whose Lebesgue number is at least $R$.
\end{defi}
\begin{lemma}
\label{dimension}
Let $X$ be a metric $\Gamma$-space with finite asymptotic dimension. Then $X$ has fqFDC.
\end{lemma}
\begin{proof}
Let $k:=\asdim X$ and let $R>0$ be given. Let $\mcU$ be an at most $k$-dimensional bounded cover of $X$ whose Lebesgue number is at least $R$. For every $G\leq\Gamma$ finite let $p_G\colon X\rightarrow G\backslash X$ be the quotient map. Define $\mcU^G:=\{p_G(U)\mid U\in\mcU\}$. Then $\{\mcU^G\}_{|G|\leq n}$ is a bounded cover of $\coprod_{G\leq\Gamma,|G|\leq n}G\backslash X$ with Lebesque number at least $R$ and dimension at most $n(k+1)$. Therefore, $\coprod_{G\leq\Gamma,|G|\leq n}G\backslash X$ has asymptotic dimension at most $n(k+1)$.\\
By \cite[Theorem 4.1]{fdc} a space with finite asymptotic dimension has FDC.
\end{proof}
\section{Permanence for fqFDC}
\label{perm}
We begin by recalling some elementary concepts from coarse geometry.
\begin{defi}
\label{coarse}
A map $f\colon X\rightarrow Y$ is
\begin{itemize}
\item \emph{uniformly expansive} if there exists a non-decreasing function \[\rho\colon [0,\infty)\rightarrow[0,\infty)\] such that for every $x, y \in X$ with $d(x,y)<\infty$
\begin{equation*}
d(f(x), f(y))\leq \rho(d(x, y)),
\end{equation*}
\item \emph{effectively proper} if there exists a proper non-decreasing function \[\delta\colon [0,\infty)\rightarrow[0,\infty)\] such that for every $x, y \in X$ with $d(x,y)<\infty$
\begin{equation*}
\delta(d(x, y))\leq d(f(x), f(y)),
\end{equation*}
\item a \emph{coarse embedding} if it is both uniformly expansive and effectively proper,
\item a \emph{coarse equivalence} if it is a coarse embedding and there exists a coarse embedding $g\colon Y\to X$ and $C>0$ such that for all $x\in X,y\in Y$ \[d((g\circ f)(x),x)\leq C,\quad d((f\circ g)(y),y)\leq C.\]
\item \emph{metrically coarse} if it is uniformly expansive and proper. If $X$ is proper and the metric on $X$ is finite, then $f$ is metrically coarse if it is a coarse embedding.
\end{itemize}
A metrically coarse homotopy between proper continuous maps is called a \emph{metric homotopy}.
\end{defi}
\begin{rem}~
\begin{itemize}
\item In the case where the metric spaces have a finite metric these definitions coincide with the common definitions but we allow effectively proper maps to map points with infinite distance close together.
\item Let $X=\bigcup_{i\in I}X_i,Y=\bigcup_{j\in J}Y_j$, where different components have infinite distance and the metric restricted to each $X_i$ resp. $Y_i$ is finite. Then viewing $X$ and $Y$ as families of metric spaces $\{X_i\}_{i\in I},\{Y_j\}_{j\in J}$ the above definitions coincide with those from \cite{rigidity}.
\end{itemize}
\end{rem}
Recall the following permanence properties of FDC from \cite{fdc}. We will restate them using only spaces instead of families.
\begin{lemma}[{Coarse Invariance \cite[3.1.3]{fdc}}]
\label{3.1.3}
Let $X,Y$ be metric spaces. If there is a coarse embedding from $X$ to $Y$ and $Y$ has FDC, then so does $X$. In particular if $X$ has FDC, then for any family of subspaces $\{X_i\}_{i\in I}$ the space $\coprod_{i\in I} X_i$ has FDC.
\end{lemma}
\begin{thm}[{Fibering \cite[3.1.4]{fdc}}]
\label{3.1.4}
Let $X,Y$ be metric spaces and $f\colon X\rightarrow Y$ uniformly expansive. Assume $Y$ has finite decomposition complexity and that for every $R>0$ and every family of subspaces $\{Z_i\}_{i\in I}$ of $Y$ with $\diam Z_i\leq R$ for all $i\in I$ also the space $\coprod_{i\in I}f^{-1}(Z_i)$ has FDC. Then $X$ has FDC.
\end{thm}
\begin{thm}[{Finite Union \cite[3.1.7]{fdc}}]
\label{3.1.7}
Let $X$ be a metric space expressed as a union $X=\bigcup_{i=1}^nX_i$ of finitely many subspaces. If the space $\coprod_{i=1}^nX_i$ has FDC, so does $X$.
\end{thm}
The following are generalizations to fqFDC of the above permanence properties.
\begin{lemma}[Coarse Invariance]
\label{coarse invariance}
Let $X,Y$ be metric $\Gamma$-spaces. If there exists a \mbox{$\Gamma$-equivariant} coarse embedding from $X$ to $Y$ and $Y$ has fqFDC, then so does $X$. In particular if $X$ has fqFDC, then for any family of subspaces $\{X_i\}_{i\in I}$ the space $\coprod_{i\in I}X_i$ has fqFDC.
\end{lemma}
\begin{proof}
Let $\rho,\delta$ be given as in \cref{coarse}. Then for all finite $G\leq \Gamma$ the map $\bar f\colon G\backslash X\rightarrow G\backslash Y$ fulfills the following for all $\bar x,\bar y\in G\backslash X$:
\[d(\bar f(\bar x),\bar f(\bar y))=\min_{g\in G}d(f(x),gf(y))\leq \min_{g\in G} \rho(d(x,gy))=\rho(d(\bar x,\bar y))\]
and
\[\delta(d(\bar x,\bar y))=\min_{g\in G}\delta(d(x,gy))\leq\min_{g\in G}d(f(x),f(gy))=d(\bar f(\bar x),\bar f(\bar y)).\]
So for every $k>0$ the map \[F\colon \coprod_{G\leq \Gamma,|G|\leq k}G\backslash X\rightarrow \coprod_{G\leq\Gamma,|G|\leq k}G\backslash Y\] is a coarse embedding as well. Since $\coprod_{G\leq\Gamma,|G|\leq k} G\backslash Y$ has FDC by assumption, so does $\coprod_{G\leq\Gamma,|G|\leq k} G\backslash X$ by \cref{3.1.3}.
\end{proof}
\begin{lemma}[Fibering]
\label{fibering}
Let $X,Y$ be metric $\Gamma$-spaces, $f\colon X\rightarrow Y$ uniformly expansive and $\Gamma$-equivariant. Assume $Y$ has fqFDC and for all $R>0$ the space $\coprod_{y\in Y}f^{-1}(B_R(y))$ has fqFDC (as a subspace of $\coprod_{y\in Y}X$ with componentwise $\Gamma$-action). Then $X$ has fqFDC.
\end{lemma}
\begin{proof}
As above the induced map $F\colon\coprod_{|G|\leq k}G\backslash X\rightarrow \coprod_{|G|\leq k}G\backslash Y$ is uniformly expansive and $\coprod_{|G|\leq k}G\backslash Y$ has FDC by assumption. By \cref{3.1.4} it suffices to show that for every $R>0$ and every family $\{Z_i\}_{i\in I}$ of subspaces of $\coprod_{|G|\leq k}G\backslash Y$ with $\diam Z_i<R$ for all $i\in I$ the space $\coprod_{i\in I} F^{-1}(Z_i)$ has FDC. Because of the bound on the diameter for all $i\in I$ there exists a $G_i\leq\Gamma$ with $|G_i|\leq k$ and $Z_i\subseteq G_i\backslash Y$. For all $i\in I$ let $pr_i\colon Y\rightarrow G_i\backslash Y$ be the projection and choose $y_i\in pr_i^{-1}(Z_i)$. Define $Z_i':=pr_i^{-1}(Z_i)\cap B_R(y_i)$. Then $pr_i(Z_i')=Z_i$ and $\diam Z_i'<R$.\\
Let $pr_i'\colon X\rightarrow G_i\backslash X$ be the projection and $f_i\colon G_i\backslash X\rightarrow G_i\backslash Y$ the map induced by $f$. Then
\[pr_i'(f^{-1}(Z_i'))=f_i^{-1}(pr_i(Z_i'))=f_i^{-1}(Z_i),\]
and since $\coprod_{i\in I}f^{-1}(Z_i')\subseteq \coprod_{y\in Y}f^{-1}(B_R(y))$ has fqFDC by assumption, in particular $\coprod_{i\in I}pr_i'(f^{-1}(Z_i'))=\coprod f_i^{-1}(Z_i)$ has FDC. But this is the same as $\coprod_{i\in I} F^{-1}(Z_i)$.
\end{proof}
\begin{lemma}[Finite Union]
\label{finunion}
Let $X$ be a metric $\Gamma$-space, $Y\subseteq X$, for $i\in\{0,..,k\}$ let $X_i\subseteq X$ have fqFDC and assume $Y\subseteq \bigcup_{i=1}^k X_i$. Then $Y$ has fqFDC.
\end{lemma}
\begin{proof}
In the case $k=2$ the space $\coprod_{|G|\leq n} G\backslash GY$ decomposes over a subspace of $\left(\coprod_{|G|\leq n}G\backslash GX_1\right)\amalg \left(\coprod_{|G|\leq n}G\backslash GX_2\right)$, which has FDC by assumption. Therefore, $\coprod_{|G|\leq n}G\backslash GY$ has FDC. The general case follows by induction.
\end{proof}
\begin{rem}
It is well known that also a finite union of spaces with finite asymptotic dimension has again finite asymptotic dimension, see \cite[Proposition 9.13]{roe}.
\end{rem}
\section{Groups with fqFDC}
\label{groups}
On every group $\Gamma$ there exists a finite proper left invariant metric. For any two finite proper left invariant metrics $d,d'$ on $\Gamma$ the identity map $(\Gamma,d)\rightarrow (\Gamma,d')$ is a coarse embedding. For finitely generated groups this is shown in \cite[Proposition 1.15]{roe} and the proof is essentially the same for proper left invariant metrics in general. Therefore, the following definition does not depend on the chosen metric.
\begin{defi}
A group $\Gamma$ has \emph{fqFDC} if $(\Gamma,d)$ has fqFDC for a finite proper left invariant metric $d$.
\end{defi}
\begin{rem}
For any subgroup $G\leq\Gamma$ and any finite proper left invariant metric $d$ on $\Gamma$ the function
\[d_{G\backslash\Gamma}(G\gamma,G\gamma'):=\inf_{g\in G}d(g\gamma,\gamma')=\min_{g\in G}d(g\gamma,\gamma')\]
defines a finite proper metric on $G\backslash\Gamma$.\\
If we have a finite proper left invariant metric on $\Gamma$ and a normal subgroup $K\trianglelefteq \Gamma$, we will always consider this metric on $\Gamma/K=K\backslash\Gamma$. This metric is again left invariant.\\
If we talk about fqFDC in context of a group $\Gamma$, we will always mean $\Gamma$ with a chosen finite proper left invariant metric.
\end{rem}
\begin{lemma}
\label{lem:preimage}
Let $K\leq \Gamma$ be normal and $\pi\colon \Gamma\to \Gamma/K$ be the quotient map. For every $r>0$ and $\gamma\in\Gamma$ we have that
\[\pi^{-1}(B_r(\gamma K))=\gamma KB_r(e)=\gamma B_r(e)K.\]
\end{lemma}
\begin{proof}
We have
\[\pi^{-1}(B_r(\gamma K))=(\gamma K)^r=\bigcup_{k\in K}B_r(\gamma k)=\bigcup_{k\in K} \gamma kB_r(e)=\gamma KB_r(e).\]
\end{proof}
To prove some permanence properties for groups we first need a stronger version of fqFDC.
\begin{defi}
A group $\Gamma$ has \emph{strong fqFDC} if for all groups $\Gamma'$ which contain $\Gamma$ as a normal subgroup and all $k\in\bbN$ the space $\coprod_{H\leq\Gamma',|H|\leq k}H\backslash H\Gamma$ has FDC.
\end{defi}
\begin{lemma}
\label{ext}
Strong fqFDC is closed under extensions $K\rightarrow \Gamma\rightarrow Q$ where $K$ is a characteristic subgroup of $\Gamma$.
\end{lemma}
\begin{proof}
Let $K\trianglelefteq \Gamma$ be characteristic, $Q:=\Gamma/K$ and $\Gamma\trianglelefteq \Gamma'$. Assume $K$ and $Q$ have strong fqFDC.\\
Since $K$ is a characteristic subgroup of $\Gamma$ the group $K$ is normal in $\Gamma'$ and $Q$ is normal in $\Gamma'/K$. So we get a uniformly expansive map 
\[\coprod_{H\leq\Gamma',|H|\leq k}H\backslash H\Gamma\longrightarrow \coprod_{H\leq\Gamma',|H|\leq k}(H/H\cap K)\backslash (H/H\cap K)Q\]
and $\coprod_{H\leq\Gamma',|H|\leq k}(H/H\cap K)\backslash (H/H\cap K)Q$ has FDC because $Q$ has strong fqFDC. So by \cref{3.1.4} and \cref{lem:preimage} it suffices to show that for all $r>0$ the space $\coprod_{|H|\leq k,q\in\Gamma}H\backslash HqB_r(e)K$ has FDC. We have
\[\coprod_{|H|\leq k,q\in\Gamma}H\backslash HqB_r(e)K\subseteq\bigcup_{\gamma\in B_r(e)}\coprod_{|H|\leq k,q\in\Gamma} H\backslash Hq\gamma K=\bigcup_{\gamma\in B_r(e)}\coprod_{|H|\leq k,q\in\Gamma}q\gamma H^{q\gamma}\backslash H^{q\gamma}K\]
where $H^{q\gamma}=(q\gamma)^{-1}Hq\gamma$.
Now the lemma follows by \cref{3.1.7} and the assumption that $K$ has strong fqFDC.
\end{proof}
\begin{lemma}
Let $\Gamma$ be the direct union of groups $\Gamma_i$ having finite asymptotic dimension. Than $\Gamma$ has strong fqFDC.
\end{lemma}
\begin{proof}
Let $k\in\bbN$ and $\Gamma\trianglelefteq\Gamma'$ be given. For every $r>0$ and $H\leq\Gamma'$ with $|H|\leq k$ define
\[Z_H:=\langle B_r(e)\cap H\Gamma\rangle.\]
Since
\[H\Gamma=\coprod^{r-disj}_{\gamma Z_H\in H\Gamma/Z_H}\gamma Z_H,\]
in particular
\[H\Gamma=\coprod^{r-disj}_{H\gamma Z_H\in H\backslash H\Gamma/Z_H}H\gamma Z_H.\]
Therefore, $\coprod_{|H|\leq k}H\backslash H\Gamma$ $r$-decomposes over $\coprod_{\gamma\in \Gamma,|H|\leq k}H\backslash H\gamma Z_H$.\\
So it suffices to show that $\coprod_{\gamma\in \Gamma,|H|\leq k}H\backslash H\gamma Z_H$ has FDC for every $r>0$.\\
Let $i$ be such that $B_{2(k+1)r}(e)\cap \Gamma\subseteq \Gamma_i$.\\
\\
\textbf{Claim:} $Z_H\cap\Gamma\leq\Gamma_i$ for all $|H|\leq k$.\\
\\
Using this we conclude that $\coprod_{|H|\leq k}Z_H\cap\Gamma$ has finite asymptotic dimension. Furthermore, $Z_H/Z_H\cap\Gamma\leq H\Gamma/\Gamma\cong H/H\cap \Gamma$ has less or equal to $k$ elements. For every $H$ choose $h^H_i\in Z_H$ with $Z_H=\bigcup_{i=1}^kh^H_i(Z_H\cap\Gamma)$. Since $h^H_i(Z_H\cap\Gamma)$ is isometric to $Z_H\cap\Gamma$ for every $1\leq i\leq k$ the space
$\coprod_{|H|\leq k}h^H_i(Z_H\cap\Gamma)$ has finite asymptotic dimension and, therefore, also 
\[\coprod_{|H|\leq k}Z_H=\bigcup_{i=1}^k\coprod_{|H|\leq k}h_i^H(Z_H\cap \Gamma)\]
has finite asymptotic dimension.\\
Now enumerating the elements of each $H$ with $|H|\leq k$ we conclude in the same way that
\[\coprod_{\gamma\in\Gamma,|H|\leq k}H\gamma Z_H=\bigcup_{i=1}^k\coprod_{\gamma\in\Gamma,|H|\leq k}h^H_i\gamma Z_H\]
has finite asymptotic dimension.\\
If $\asdim \coprod_{\gamma\in\Gamma,|H|\leq k}H\gamma Z_H\leq m$, then $\asdim\coprod_{\gamma\in\Gamma,|H|\leq k}H\backslash H\gamma Z_H\leq k(m+1)$ (see the proof of \cref{dimension}).\\
Therefore, $\coprod_{\gamma\in\Gamma,|H|\leq k}H\backslash H\gamma Z_H$ has FDC.\\
\\
It remains to prove the above claim:\\
Let $z\in Z_H$ and let $m$ be minimal with
\[z=h_1\gamma_1\ldots h_{m}\gamma_{m}\gamma\]
for some $h_j\in H,\gamma_j\in\Gamma,\gamma\in\Gamma_i$ such that $h_j\gamma_j\in B_r(e)$ for all $j$.\\
If $m>k\geq |H|$, then there exist integers $n_1,n_2$ with $m-|H|\leq n_1<n_2\leq m$ and \[h_{n_1}\ldots h_m=h_{n_2}\ldots h_m.\] Therefore, 
\[(h_{n_1}\gamma_{n_1}\ldots h_m\gamma_m)^{-1}h_{n_2}\gamma_{n_2}\ldots h_m\gamma_m\in \Gamma\cap B_{2(k+1)r}(e)\subseteq \Gamma_i\]
and there exists $\gamma'\in\Gamma_i$ with
\[h_{n_1}\gamma_{n_1}\ldots h_m\gamma_m=h_{n_2}\gamma_{n_2} \ldots h_m\gamma_m\gamma'.\]
So $m$ is not minimal, a contradiction.\\
Let $z\in Z_H\cap\Gamma$ be represented as
\[z=h_1\gamma_1\ldots h_{m}\gamma_{m}\gamma\]
for some $h_j\gamma_j\in B_{r}(e)\cap H\Gamma,\gamma\in\Gamma_i$ with $m\leq k$.\\
Then $h_1\gamma_1\ldots h_{m}\gamma_{m}\in \Gamma\cap B_{kr}(e)\subseteq \Gamma_i$ and therefore also $z=h_1\gamma_1\ldots h_{m}\gamma_{m}\gamma\in\Gamma_i$. This proves the claim.
\end{proof}
By the classification of finitely generated abelian groups we immediately get the following:
\begin{cor}
\label{abelian}
Abelian groups have strong fqFDC.\qed
\end{cor}
Combining \cref{ext} and \cref{abelian} yields:
\begin{cor}
\label{solv}
Solvable groups have strong fqFDC.\qed
\end{cor}
To show that finitely generated linear groups have fqFDC we need the following extension property.
\begin{prop}
\label{ext2}
Let $K\rightarrow \Gamma\rightarrow Q$ be an extension and let $K$ have strong fqFDC and $Q$ fqFDC. Then $\Gamma$ has fqFDC.
\end{prop}
\begin{proof}
$\Gamma\rightarrow Q$ is uniformly expansive and $Q$ has fqFDC. So by \cref{fibering} and \cref{lem:preimage} it suffices to show that for all $r>0$ the space
\[\coprod_{g\in\Gamma}gB_r(e)K\subseteq \bigcup_{\gamma\in B_r(e)}\coprod_{g\in\Gamma}g\gamma K\]
has fqFDC. This follows from \cref{finunion} and the fact that for every $k\in \bbN,\gamma\in B_r(e)$ the space \[\coprod_{|H|\leq k,g\in\Gamma}H\backslash Hg\gamma K=\coprod_{|H|\leq k,g\in\Gamma}g\gamma H^{g\gamma}\backslash H^{g\gamma}K\]
has FDC by assumption.
\end{proof}
\subsection*{Linear groups}
\begin{thm}
\label{lin}
Let $\Gamma$ be a finitely generated subgroup of $GL_n(F)$ where $F$ is a field. Then $\Gamma$ has fqFDC.
\end{thm}
This theorem is the fqFDC version of \cite[Theorem 3.1]{rigidity}. All steps in the proof hold for fqFDC as well, once we prove the following fqFDC version of \cite[Lemma 3.9]{rigidity}.
\begin{lemma}
Let $G$ be a countable discrete group. Suppose there exists a finite left invariant metric $d'$ on $G$ with the following properties:
\begin{enumerate}
\item $G$ has finite asymptotic dimension with respect to $d'$.
\item For all $r > 0$ there exists $d_r$, a finite left invariant metric on $G$, for which
\begin{enumerate}
\item $G$ has finite asymptotic dimension with respect to $d_r$,
\item $d_r$ is proper when restricted to $B_{r,d'}(e)$, the ball of radius $r$ around $e$ with respect to the metric $d'$.
\end{enumerate}
\end{enumerate}
Then $G$ has fqFDC.
\end{lemma}
Condition (b) in the lemma means precisely that $B_{s,d_r}(e)\cap B_{r,d'}(e)$ is finite for every $s > 0$. In the statement of \cite[Lemma 3.9]{rigidity} pseudo-metrics are used but if $(X,d)$ is a pseudo-metric space, then $(X,d')$ with
\[d'(x,y):=\left\{\begin{matrix}
0&\text{if }x=y\\
\max\{1,d(x,y)\}&\text{if }x\neq y
\end{matrix}\right.\]
is a metric space that is coarsely equivalent to $(X,d)$. Since finite asymptotic dimension is a coarse invariant, we can use this to pass to metrics instead of pseudo-metrics.

In the proof we will use that finite asymptotic dimension implies fqFDC\linebreak(see \cref{dimension}).
\begin{proof}
Fix a finite proper left invariant metric $d$ on $G$. By \cref{fibering}, applied to the identity map $(G,d)\rightarrow (G,d')$, it suffices to show that for every $r > 0$ the space $\coprod_{g\in G}B_{r,d'}(g)=\coprod_{g\in G}gB_{r,d'}(e)$ has finite asymptotic dimension when equipped with the metric $d$. Because all spaces $gB_{r,d'}(e)$ are isometric to $B_{r,d'}(e)$ it suffices to show that $B_{r,d'}(e)$ has finite asymptotic dimension.\\
Let $r > 0$. Pick $d_{2r}$ as in the assumptions. The ball $B_{r,d'}(e)\subseteq G$  has finite asymptotic dimension with respect to the metric $d_{2r}$.\\ 
Thus, it remains to show that the metrics $d$ and $d_{2r}$ on $B_{r,d'}(e)$ are coarsely equivalent.\\
Since balls in $G$ with respect to $d$ are finite, we easily see that for every $s$ there exists $s'$ such that if $d(g,h)\leq s$, then $d_{2r}(g,h)\leq s'$; this holds for every $g$ and $h$ in $G$. Conversely, for every $s$ the set $B_{2r,d'}(e)\cap B_{s,d_{2r}}(e)$ is finite by assumption, and we obtain $s'$ such that for every $g$ in this set $d(g,e)\leq s'$. If now $g,h\in B_{r,d'}(e)$ are such that $d_{2r}(g,h)\leq s$, then $g^{-1}h\in B_{s,d_{2r}}(e)$ and $d(g,h)= d(g^{-1}h,e)\leq s'$.
\end{proof}
To generalize this to arbitrary commutative rings we need Lemma 5.2.3 from \cite{fdc}:
\begin{lemma}
Let $R$ be a finitely generated commutative ring with unit and let $\mcn$ be the nilpotent radical of $R$,
\[\mcn=\{r\in R\mid \exists n:r^n=0\}.\]
The quotient ring $S=R/\mcn$ contains a finite number of prime ideals $\mcp_1,..,\mcp_n$ such that the diagonal map 
\[S\rightarrow S/\mcp_1\oplus\ldots\oplus S/\mcp_n\]
embeds $S$ into a finite direct sum of domains.
\end{lemma} 
The next theorem is the fqFDC version of \cite[Theorem 5.2.2]{fdc}. We need to assume that $\Gamma$ is finitely generated because we do not know if fqFDC is closed under unions.
\begin{thm}
Let $R$ be a commutative ring with unit and let $\Gamma$ be a finitely generated subgroup of $GL(n,R)$, then $\Gamma$ has fqFDC.
\end{thm}
\begin{proof}
Because $\Gamma$ is finitely generated we can assume that $R$ is finitely generated as well. With $\mcn$ and $S$ as in the previous lemma, we have an exact sequence
\[1\rightarrow I+M_n(\mcn)\rightarrow GL_n(R)\rightarrow GL_n(S)\to 1\]
in which $I+M_n(\mcn)$ is nilpotent and therefore has strong fqFDC by \cref{solv}. In the notation of the previous lemma we have embeddings
\[GL_n(S)\hookrightarrow GL_n(S/\mcp_1)\times\ldots\times GL_n(S/\mcp_n)\hookrightarrow GL_n(Q(S/\mcp_1))\times\ldots\times GL_n(Q(S/\mcp_n))\]
where $Q(S/\mcp_i)$ is the quotient field of $S/\mcp_i$.\\
So the quotient has fqFDC by \cref{lin}.\\
Now the theorem follows from \cref{ext2}.
\end{proof}
\section{Controlled Algebra}
Let $X$ be a proper metric space, $\mcA$ an additive category and $\Gamma$ a group. The definition of a geometric module in this article is a slight variation of the definitions in \cite{bartels} and \cite{k-theory}. The first definition of geometric groups appeared in \cite{MR0242151} and of geometric modules in \cite{MR549490} and \cite{MR669423}. The first definition of continuous control is in \cite{MR1277522}.
\begin{defi}
Let $Z:=\Gamma\times X\times[0,1)$. A \emph{geometric $\mcA$-module} $M$ over $X$ is given by a sequence of objects $(M_z)_{z \in Z}$ in $\mcA$, subject to the following conditions:
\begin{enumerate}
\item The image of $\supp(M)=\{z\in Z\mid M_z\neq 0\}$ under the projection \[Z\rightarrow X\times[0,1)\] is locally finite.
\item \label{objects2} For every $x\in X,t\in[0,1)$ the set $\supp(M)\cap(\Gamma\times\{x\}\times\{t\})$ is finite.
\end{enumerate}
A \emph{morphism} $\phi\colon M\rightarrow N$ between geometric modules $M,N$ is a sequence\linebreak$(\phi_{x,y}\colon M_y\rightarrow N_x)_{(x,y)\in Z^2}$ of morphisms in $\mcA$, subject to the following conditions:
\begin{enumerate}
\item $\phi$ is \emph{continuously controlled at 1}, i.e. for each $x\in X$ and each neighborhood $U$ of $(x,1)$ in $X\times[0,1]$ there exists a neighborhood $V$ of $(x,1)$ in $X\times [0,1]$ such that for all $\gamma,\gamma'\in\Gamma,v\in V,y\notin U,$ $\phi_{(\gamma,v),(\gamma',y)}=\phi_{(\gamma',y),(\gamma,v)}=0$.
\item For every $z\in Z$ the set $\{z'\in Z\mid \phi_{z,z'}\neq 0\text{ or }\phi_{z',z}\neq0\}$ is finite.
\item $\phi$ is \emph{R-bounded} for some $R>0$, i.e. $d(x,x')>R$ implies $\phi_{(\gamma,x,t),(\gamma',x',t')}=0$ for all $\gamma,\gamma'\in\Gamma,x,x'\in X,t,t'\in[0,1)$.
\end{enumerate}
Let $\mcA_\Gamma(X)$ denote the category of geometric $\mcA$-modules over $X$ and their morphisms. The composition of morphisms is given by matrix multiplication. $\mcA_\Gamma(X)$ is an additive category with pointwise addition.
\end{defi}
\begin{rem}
Let $\mcA_c(X)\subseteq\mcA_\Gamma(X)$ be the full additive subcategory with objects having support in $\{e\}\times X\times[0,1)$. This coincides with the definition of $\mcA_c(X)$ in \cite{k-theory}.\\
The inclusion $\mcA_c(X)\hookrightarrow\mcA_\Gamma(X)$ is an equivalence because of condition (2) on the objects of $\mcA_\Gamma(X)$.
\end{rem}
\begin{nota}
For $Y\subseteq X$ and $M\in\mcA_\Gamma(X)$ we define $M|_Y$ by
\[(M|_Y)_{(\gamma,x,t)}:=\left\{\begin{matrix}
M_{(\gamma,x,t)}&\text{if }x\in Y\\0&\text{else}
\end{matrix}\right.,\quad\forall \gamma\in\Gamma,x\in X,t\in[0,1)\]
\end{nota}
\begin{defi}
\label{def:kar}
Let $\mcU$ be a full additive subcategory of an additive category $\mcA$. We say that $\mcA$ is \emph{$\mcU$-filtered} if for all $A\in \mcA, U\in\mcU$ and all morphisms $f\colon A\to U,$ $g\colon U\to A$, we have a decomposition $A\cong A^-\oplus A^+$ with $A^-\in\mcU$ and factorizations
\[\xymatrix{U\ar[rr]^g\ar[rd]&&A\\&A^-\ar[ru]^\iota&
}\quad \text{and}\quad \xymatrix{A\ar[rr]^f\ar[rd]^p&&U\\&A^-\ar[ru]&
}\]
where $\iota,p$ is the inclusion resp. the projection. Note that we need one decomposition through which both maps factor.
\end{defi}
\begin{nota}
For two decompositions $A\cong A_1\oplus A_2$ and $A\cong A_1'\oplus A_2'$ we write $A_1\subseteq A_1'$ if the map
\[A_1\xrightarrow{ \iota}A\xrightarrow{p}A_2'\]
is zero, where $\iota,p$ is the inclusion resp. the projection.
\end{nota}
%\begin{lemma}
%\label{lem:leq}
%Let $A\xrightarrow{\phi}A_1\oplus A_2$ and $A\xrightarrow{\psi}B_1\oplus B_2$ be two decompositions with $A_1\subseteq B_1$. Then we can change $\psi$ to a decomposition $\psi'\colon A\to B_1\oplus B_2$ with $A_1\subseteq B_1$ and $B_2\subseteq A_2$.
%\end{lemma}
%\begin{proof}
%Let $\iota_i^C\colon C_i\to A$ with $i\in\{1,2\},C\in\{A,B\}$ denote the inclusions and $p_i^C\colon A\to C_i$ the projections. Since $A_1\subseteq B_1$ we have that $p_2^B\iota_1^A=0$. Then \[((\iota_1^B)',(\iota_2^B)')=(\iota_1^B,\iota_2^B-\iota_1^Ap_1^B\iota_2^B)\colon B_1\oplus B_2\to A\] is an isomorphism with inverse \[((p_1^B)',(p_2^B)')=(p_1^B-p_1^B\iota_1^Ap_1^A\iota_2^Bp_2^B,p_2^B)\colon A\to B_1\oplus B_2.\]
%We have $(p_2^B)' \iota_1^A=p_2^B\iota_1^A=0$ and $p_1^A(\iota_2^B)'=p_1^A\iota_2^B-p_1^A\iota_1^Ap_1^A\iota_2^B=0$ and thus $A_1\subseteq B_1$ and $B_2\subseteq A_2$.
%\end{proof}
\begin{lemma}
If $\mcA$ is $\mcU$-filtered as defined above, then it is $\mcU$-filtered in the usual sense (\cite{karoubi}), i.e. there exists a partially ordered, directed index set $I$ and for each $A\in\mcA, i\in I$ there are decompositions $A\cong A_i^-\oplus A_i^+$ with $A^-_i\in\mcU$, such that
\begin{enumerate}
\item For every morphism $f\colon A\to U$ with $U\in\mcU$ there exists $i\in I$ such that $f$ factors as
\[\xymatrix{A\ar[rr]^f\ar[rd]^{p_i}&&U\\&A_i^-\ar[ru]&
},\]
where $p_i$ is the projection.
\item For every morphism $g\colon U\to A$ with $U\in\mcU$ there exists $i\in I$ such that $g$ factors as
\[\xymatrix{U\ar[rr]^g\ar[rd]&&A\\&A_i^-\ar[ru]^{\iota_i}&
},\]
where $\iota_i$ is the inclusion.
%\item For $i,j\in I$ with $i\leq j$ and all $A\in\mcA$ we have $A_j^+\subseteq A_i^+$ and $A_i^-\subseteq A_j^-$.
\item For all $i,j\in I$ we have
\[i\leq j\quad\Longleftrightarrow\quad \forall A\in\mcA:~A_j^+\subseteq A_i^+\text{ and }A_i^-\subseteq A_j^-.\]
\item For all $A,B\in\mcA$ and all $(i,j)\in I^2$ there exists an index $k\in I$ such that \mbox{$(A\oplus B)^+_k\subseteq A_i^+\oplus B_j^+$} and $A_i^-\oplus B_j^-\subseteq (A\oplus B)^-_k$. And for all $A,B\in\mcA$ and $k\in I$ there exists $(i,j)\in I^2$ such that $A_i^+\oplus B_j^+\subseteq (A\oplus B)_k^+$ and $(A\oplus B)_k^-\subseteq A_i^-\oplus B_j^-$.
\end{enumerate}
\end{lemma}
\begin{proof}
For $A\in \mcA$ let $I_A'$ denote the set of all possible decompositions $A\cong A^-\oplus A^+$ with $A^-\in\mcU$. We define a relation on $I_A'$ by $(A\cong A^-\oplus A^+)\leq (A\cong (A')^-\oplus (A')^+)$ if $A^-\subseteq (A')^-$ and $(A')^+\subseteq A^+$. Let $I_A$ be the quotient of $I_A'$ where we identify $i,j\in I_A'$ if $i\leq j$ and $j\leq i$. With this relation $I_A$ is a poset. For every $i\in I_A$ we choose a preimage $(A\cong A^-\oplus A^+)\in I_A'$ and define the decomposition $A\cong A_i^-\oplus A_i^+$ to be $A\cong A^-\oplus A^+$. This way for $i,j\in I_A$ we have
\[i\leq j\quad\Longleftrightarrow\quad A_j^+\subseteq A_i^+\text{ and }A_i^-\subseteq A_j^-.\]

Let $A\in\mcA,i,j\in I_A$. By assumption there exists a decomposition $A\cong V\oplus A'$  with $V\in\mcU$ such that the map $A_i^-\oplus A_j^-\xrightarrow{\iota_i+\iota_j}A$ factors as
\[\xymatrix{A_i^-\oplus A_j^-\ar[rr]^{\iota_i+\iota_j}\ar[rd]&&A\\&V\ar[ru]^\iota&}\]
and the map $A\xrightarrow{(p_i,p_j)}A_i^-\oplus A_j^-$ factors as
\[\xymatrix{A\ar[rr]^{(p_i,p_j)}\ar[rd]^p&&A_i^-\oplus A_j^-\\&V\ar[ru]&},\]
where $\iota,p$ is the inclusion resp. the projection.
This implies $A_i^-,A_j^-\subseteq V$ and $A'\subseteq A_i^+,A_j^+$. Therefore, $i,j\leq[A\cong V\oplus A']$ and $I_A$ is directed. Define $I:=\prod_{A\in\mcA}I_A$ with $(i_A)_A\leq (i'_A)_A$ if $i_A\leq i'_A$ for all $A\in\mcA$. For $(i_{A'})_{A'\in\mcA}\in I$ and $A\in\mcA$ define the decomposition $A\cong A_i^-\oplus A_i^+$ to be the decomposition $A\cong A_{i_A}^-\oplus A_{i_A}^+$. With this definition $I$ is a directed poset and axioms (1)-(3) are satisfied. It remains to prove axiom (4).

%Let $A,B\in\mcA$ and fix $A\oplus B$. For each $C\in\mcA$ with $C\neq A\oplus B$ choose $i^\circ_C\in I_C$.
Let $A,B\in\mcA,(i,j)\in I$ then there exists $k\in I$ such that $A_i^-\oplus B_j^-\to A\oplus B$ and $A\oplus B\to A_i^-\oplus B_j^-$ factor over $(A\oplus B)_k^-$. As above this implies that we have $(A\oplus B)^+_k\subseteq A_i^+\oplus B_j^+$ and $A_i^-\oplus B_j^-\subseteq (A\oplus B)_k^-$. %Define $f:I^2\to I$ by choosing $k\in I_{A\oplus B}$ as above and defining
%\[f(i,j)_C=\left\{\begin{matrix}k&C=A\oplus B\\i^\circ_C&else\end{matrix}\right.\]
%Let again $A,B\in\mcA$ and now for each $C\neq A,B$ choose $i^\circ_C\in I_C$.
For $A,B\in\mcA,k\in I$ there exist $i,j\in I$ such that $(A\oplus B)_k^-\to A\oplus B\to A$ and $A\to A\oplus B\to (A\oplus B)_k^-$ factor over $A_i^-$ and $(A\oplus B)_k^-\to A\oplus B\to B$ and $B\to A\oplus B\to (A\oplus B)_k^-$ factor over $B_j^-$. Again this implies that $A_i^+\oplus B_j^+\subseteq (A\oplus B)_k^+$ and $(A\oplus B)_k^-\subseteq A_i^-\oplus B_j^-$. 
%Define $f':I\to I^2$ by choosing $i\in I_A, j\in I_B$ as above and defining
%\[f'(k)_C=\left\{\begin{matrix}i&C=A\\j&C=B\\i^\circ_C&else\end{matrix}\right.\]
%By construction of $I$ the maps $f,f'$ are poset maps.
\end{proof}~
\begin{rem}~
\label{rem:kar}
\begin{enumerate}
\item Obviously also the other direction holds, i.e. \cref{def:kar} is equivalent to the usual definition of $\mcU$-filtered.
\item If $\mcU$ is a full additive subcategory of $\mcA$, then the quotient $\mcA/\mcU$ is defined as the category having the same objects as $\mcA$ and morphisms being equivalence classes of morphisms of $\mcA$, where $f,g\colon A\rightarrow B$ are equivalent if $f-g$ factors through some $U\in\mcU$.
\item Let $\mcA$ be an idempotent complete additive category and $\mcU$ a full subcategory. Suppose for all objects $A\in\mcA, U\in\mcU$ and every morphism $f\colon A\to \mcU$ there is a decomposition $A=A^-\oplus A^+$ with $A^-\in\mcU$ such that $f$ factors as
\[\xymatrix{A\ar[rr]^f\ar[rd]^p&&U\\&A^-\ar[ru]&
}.\]
In \cite[Theorem 2.1]{MR2079996} Schlichting proves that the sequence $\mcU\to\mcA\to\mcA/\mcU$ induces a homotopy fibration sequence in $K$-theory. 
\item With \cref{def:kar} it is easy to see, that if $\mcU_1\subseteq \mcU_2\subseteq \mcA$ are full subcategories such that $\mcU_1\subseteq\mcU_2$ is an equivalence of categories and $\mcA$ is $\mcU_2$-filtered, then $\mcA$ is also $\mcU_1$-filtered.
\item For controlled categories it is usually shown that for $Y\subseteq X$ the category $\mcA_c(X)$ is $\mcA_c^+(Y)$-filtered, where $\mcA_c^+(Y)$ is the full subcategory consisting of all objects supported "near" $Y$ and that the inclusion $\mcA_c(Y)\to\mcA_c^+(Y)$ is an equivalence of categories, see for example \cite[Lemma 3.6]{squeezing}. By the above this shows that $\mcA_c(X)$ is also $\mcA_c(Y)$-filtered.
\end{enumerate}
\end{rem}
%We now recall the definition of a Karoubi filtration \cite[Definition 1.27]{MR1341817}.
%\begin{defi}
%Let $\mcU$ be a full subcategory of an additive category $\mcA$. We say that $\mcA$ is $\mcU$-filtered if every object $A\in\mcA$ has a family of decompositions $\{A\cong E_i\oplus U_i\}$ with $E_i\in\mcA,U_i\in\mcU$ such that
%\begin{enumerate}
%\item for each $A\in\mcA$ the decompositions form a filtered poset under the partial order $E_i\oplus U_i\leq E_j\oplus U_j$ whenever $E_j$ is a direct summand of $E_i$ and $U_i$ is a direct summand of $U_j$.
%\item for each $A\in\mcA$ and $U\in \mcU$ every map $f\colon A\rightarrow U$ factors as \[A\cong E_i\oplus A_i\rightarrow E_i\rightarrow U\] for some $i$.
%\item for each $A\in\mcA$ and $U\in\mcU$ every map $f\colon U\rightarrow A$ factors as \[U\rightarrow U_i\rightarrow E_i\oplus U_i\cong A\] for some $i$.
%\item for each $A,B\in \mcA$ the filtration of $A\oplus B$ is equivalent to the sum of the filtrations $\{A=E_i\oplus U_i\}$ and $\{B=F_j\oplus V_j\}$, i.e. to $\{E_i\oplus F_j\oplus U_i\oplus V_j\}$.
%\end{enumerate}
%\end{defi}
By $K$-theory of an additive category we will always mean the non-connective \mbox{$K$-theory} spectrum \cite{MR802790}.
\begin{thm}[{\cite[Theorem 1.28]{MR1341817}}]
If $\mcA$ is $\mcU$-filtered, then 
\[\bbK(\mcU)\rightarrow \bbK(\mcA)\rightarrow\bbK(\mcA/\mcU)\]
is a homotopy fibration sequence.
\end{thm}
\begin{defi}[{\cite[Definition 1.1]{MR802790}}]
An additive category $\mcA$ is said to be \emph{filtered} if there is an increasing filtration
\[F_0(A,B)\subseteq F_1(A,B)\subseteq\ldots\subseteq F_n(A,B)\subseteq\ldots\]
on $\hom(A,B)$ for every pair of objects $A,B\in\mcA$. Each $F_i(A,B)$ has to be an additive subgroup of $\hom(A,B)$ and we must have $\bigcup_{i\in\bbN} F_i(A,B)=\hom(A,B)$. We require the zero and identity maps to be in the zeroth filtration degree and for $f\in F_i(A,B)$ and $g\in F_j(B,C)$ the composition $g\circ f$ to be in $F_{i+j}(A,C)$. If $f\in F_i(A,B)$, we say that $f$ has \emph{(filtration) degree} $i$.
\end{defi}
\begin{rem}
We do not demand projections $A\oplus B\rightarrow A$ and inclusions $A\rightarrow A\oplus B$ to have degree zero because then we either have to specify choices for the sums or every isomorphism would have degree zero.
\end{rem}
\begin{defi}
For filtered additive categories $\{\mcA_i\}_{i\in I}$ we define $\prod_{i\in I}^{bd}\mcA_i$ to be the subcategory of $\prod_{i\in I}\mcA_i$ containing all objects and those morphisms $\phi=\{\phi_i\}_{i\in I}$ such that for some $n\in \bbN$, the morphism $\phi_i$ has degree $n$ for all $i\in I$.\\
For a metric space $X$ the categories $\mcA_\Gamma(X)$ and $\mcA_c(X)$ are filtered by defining a morphism $f$ to be of degree $n$ if it is $n$-bounded.
\end{defi}
\begin{prop}
\label{prop:5.8}
Let $\{Y_i\subseteq X_i\}_{i\in I}$ be a family of metric spaces with subspaces. Then the inclusion $\prod_{i\in I}^{bd}\mcA_\Gamma(Y_i)\hookrightarrow \prod_{i\in I}^{bd}\mcA_\Gamma(X_i)$ is a Karoubi filtration. We will denote the quotient by $\prod_{i\in I}^{bd}\mcA_\Gamma(X_i,Y_i)$. The same holds when $\mcA_\Gamma$ is replaced by $\mcA_c$.
\end{prop}
\begin{proof}
Let $Z_i:=\Gamma\times X_i\times [0,1)$ and let  $\{A_i\}_{i\in I}\in\prod_{i\in I}^{bd}\mcA_\Gamma(X_i), \{U_i\}_{i\in I}\in\prod_{i\in I}^{bd}\mcA_\Gamma(Y_i)$. Let $\phi=\{\phi_i\colon A_i\rightarrow U_i\}_{i\in I}$ be a morphism from $\{A_i\}_{i\in I}$ to $\{U_i\}_{i\in I}$ and let \linebreak\mbox{$\psi=\{\psi_i\colon U_i\to A_i\}_{i\in I}$} be a morphism from $\{U_i\}_i$ to $\{A_i\}_i$. Under the projection \[p_i\colon Z_i\times Z_i\rightarrow X_i\times X_i\] the subspace $W_i:=\{(z,z')\in Z_i\times Z_i\mid \phi_{z',z}\neq 0\wedge \psi_{z,z'}\neq 0\}$ has image inside $X_i\times Y_i$. Let $V_i:=\{x\in X_i\mid \exists y:(x,y)\in p_i(W_i)\}$. Define $\rho_i\colon V_i\rightarrow Y_i$ by choosing for every $v\in V_i$ a $y$ with $(v,y)\in p_i(W_i)$. Furthermore, define $\{B_i\}_{i\in I}\in\prod^{bd}_{i\in I}\mcA_\Gamma(Y_i)$ by \[(B_i)_{(\gamma,y,t)}:=\bigoplus_{x\in\rho_i^{-1}(y)}(A_i)_{(\gamma,x,t)},\] 
for all $i\in I, \gamma\in\Gamma, y\in Y_i$ and $t\in[0,1)$.

Since $\phi,\psi$ are morphisms in $\prod_{i\in I}^{bd}\mcA_\Gamma(X_i)$ also the isomorphism $\{B_i\}_{i\in I}\cong \{A_i|_{V_i}\}_{i\in I}$ satisfies all control conditions and we get a decomposition $\{A_i\}_i\cong \{B_i\}_i\oplus \{A_i|_{X_i\setminus V_i}\}_i$. The morphisms $\phi$ and $\psi$ factor through $\{A_i|_{V_i}\}_i$ by definition, so they factor through $\{B_i\}_i$ as well. The same proof holds for $\mcA_c$ instead of $\mcA_\Gamma$.
\end{proof}
\begin{defi}
If $\Gamma$ acts on $X$ and $\mcA$, then $\Gamma$ acts on the category $\mcA_\Gamma(X)$ by $(\gamma M)_x:=\gamma (M_{\gamma^{-1}x})$ and the corresponding action on the morphisms. For a subgroup $G\leq\Gamma$ and $Y_i\subseteq X_i$ $G$-invariant let $\prod_{i\in I}^{bd}\mcA^G_\Gamma(X_i,Y_i)$ be the corresponding fixed point category of $\prod_{i\in I}^{bd}\mcA_\Gamma(X_i,Y_i)$. This is equivalent to the quotient of $\prod_{i\in I}^{bd}\mcA^G_\Gamma(X_i)$ by $\prod_{i\in I}^{bd}\mcA^G_\Gamma(Y_i)$.
\end{defi}
\begin{defi}
For a continuous map $p\colon X\rightarrow X'$ let $\mcA_{\Gamma,p}(X)$ be the full subcategory of $\mcA_\Gamma(X)$ with objects having support in $\Gamma\times p^{-1}(K)\times[0,1)$ for some compact subspace $K\subseteq X'$. For some $G\leq\Gamma$ acting on $X$ let $\mcA^G_{\Gamma,p}(X)$ be the fixed point category as above.\\
Furthermore, let $\mcA^G_{\Gamma,p}(X)_0$ be the full subcategory of $\mcA^G_{\Gamma,p}(X)$ with the following condition on the support of the objects:
\begin{itemize}
\item For every object $M$ the limit points of the image of $\supp(M)$ under\linebreak$Z\rightarrow X\times[0,1)$ are disjoint from $X\times\{1\}$.
\end{itemize}
The inclusion $\mcA^G_{\Gamma,p}(X)_0$ into $\mcA^G_{\Gamma,p}(X)$ is a Karoubi filtration.\\
Define $\mcA_{\Gamma,p}^G(X)^\infty$ to be the quotient of $\mcA^G_{\Gamma,p}(X)$ by $\mcA^G_{\Gamma,p}(X)_0$.
\end{defi}
\begin{rem}[{\cite[(5.15)]{bartels}}]
\label{rem:functorial}
The categories $\mcA_c(X) $defined above are functorial in $X$ for continuous metrically coarse maps $f\colon X\rightarrow Y$ if two such maps are metrically homotopic (see \cref{perm}) then they induce homotopic maps of the $K$-theory spectra.

The categories $\mcA_{\Gamma,p}(X)$ are functorial in $X$ for commutative diagrams of the form
\[\xymatrix{X\ar[r]^f\ar[d]^p&Y\ar[d]^{p'}\\X'\ar[r]^{f'}&Y'}\]
where $f$ is continuous metrically coarse and $f'$ is continuous. If two such diagrams are homotopic, i.e. we have a metric homotopy $H$ between $f_1,f_2$ and a homotopy $H'$ between $f_1',f_2'$ such that $p'H=H'p$, then $f_1,f_2$ induce homotopic maps
\[\bbK(\mcA_{\Gamma,p}(X))\to \bbK(\mcA_{\Gamma,p}(X)).\]
The same is true for $\mcA_{\Gamma,p}(X)_0$ and $\mcA_{\Gamma,p}(X)^\infty$.
\end{rem}
\begin{rem}
Let $p\colon\underbar E\Gamma\rightarrow \Gamma\backslash \underbar E\Gamma$ be the projection. The category $\mcA_{p,\Gamma}^\Gamma(\underbar E\Gamma)_0$ is equivalent to the category $\mcA[\Gamma]$ defined in \cite{MR2294225}. There the assembly map
\[H_*(\underbar E\Gamma;\bbK_\mcA)\rightarrow H_{*}(pt;\bbK_\mcA)=\pi_*\bbK(\mcA[\Gamma])\]
is also defined.\\
The controlled categories defined above can be used to describe that assembly map. More precise the boundary map
\[\pi_*\bbK\mcA_{p,\Gamma}^\Gamma(\underbar E\Gamma)^\infty\rightarrow \pi_{*-1}\bbK\mcA_{p,\Gamma}^\Gamma(\underbar E\Gamma)_0\]
is equivalent to the assembly map. Compare \cite[(5.17)]{bartels}.
\end{rem}
\section{The Rips complex}
Recall the following definitions.
\begin{defi}
A metric space $X$ has \emph{bounded geometry} if for each $R > 0$ there exists $N > 0$ such that for all $x\in X$ the ball $B_R(x)$ contains at most $N$ points.
\end{defi}
\begin{defi}
Given a metric space $X$ and a number $s > 0$, the \emph{Rips complex} $P_s(X)$ is the simplicial complex with vertex set $X$ and with a simplex
$\langle x_0,\ldots,x_n\rangle$ whenever $d(x_i,x_j)\leq s$ for all $i, j\in\{0, . . . ,n\}$.
\end{defi}
Note that if $X$ is a metric space with bounded geometry, then the Rips complex $P_s(X)$ is finite dimensional and locally finite.
We will always use the simplicial path metric on $P_s(X)$. For a definition of the simplicial path metric see \cite[Section 2]{k-theory}.\\
The following is similar to \cite[Lemma 4.10]{rigidity} and is proved in a special case in the proof of \cite[Theorem 7.8]{k-theory}
\begin{prop}
\label{compare}
Let $X$ be a metric CW-complex, $Y$ a subcomplex such that $X$ is uniformly contractible with respect to $Y$ and the cells in $Y$ have uniformly bounded diameter. Let $S\subseteq Y$ be a subspace with bounded geometry such that there exists \mbox{$R>0$} with $Y\subseteq S^R$. Then there exist continuous metrically coarse maps \mbox{$f_s\colon Y\rightarrow P_s(S)$}, $g_s\colon P_s(S)\rightarrow X$ for every $s>2R$ such that the following diagram commutes for $s'>s$
\[\xymatrix{Y\ar[r]^{f_s}\ar[rd]_{f_{s'}}&P_s(S)\ar[r]^{g_s}\ar[d]^{i_{ss'}}&X\\&P_{s'}(S)\ar[ru]_{g_{s'}}&}\]
where $i_{ss'}\colon P_s(S)\rightarrow P_{s'}(S)$ is the inclusion and such that $g_s\circ f_s$ is metrically homotopic to the inclusion $i\colon Y\hookrightarrow X$.
\end{prop}
\begin{proof}
Let $R>0$ be such that $Y\subseteq S^R$, then $\{U_{x'}:=Y\cap(B_R(x')-(S-\{x'\}))\}_{x'\in S}$ is an open covering of $Y$. Choose a partition of unity $\{\phi_{x'}\}$ subordinate to the cover $\{U_{x'}\}$ and define $f_s\colon Y\rightarrow P_s(S)$ by
\[f_s(y):=\sum_{x'\in S}\phi_{x'}(y)x'\]
for every $y\in Y$, $s>2R$.\\
Define maps $g_s\colon P_{s}(S)\rightarrow X$ by induction over $s\in\bbN$ and the simplices in $P_s(S)$. If $s=0$, $P_s(S)=S$ and $g_0$ is just the embedding $S\hookrightarrow X$.\\
Now assume $g_{s-1}$ has been defined. Let $P_s^{(k)}(S)$ be the $k$-skeleton of $P_s(S)$. Viewing $P_{s-1}(S)$ as a subcomplex of $P_s(S)$ we extend $g_{s-1}$ inductively over the subspaces $P_s^{(k)}(S)\cup P_{s-1}(S)$. Assuming $g_s$ has been defined on $P_s^{(k-1)}(S)\cup P_{s-1}(S)$ we extend over a $k$-simplex $\sigma\notin P_{s-1}(S)$ as follows:\\
Let $D=\diam(g_s(\partial \sigma))$. Since $g_s$ is metrically coarse $D$ is independent of $\sigma$ and since $S\subseteq Y$ and $g_s(S)=S$ we have $g_s(\partial\sigma)\cap Y\neq \emptyset$. Choose $y\in g_s(\partial \sigma)$. By uniform contractibility of $X$ relative to $Y$ there exists $D'$, depending only on $D$, and a null homotopy of $g_s|_{\partial\sigma}$ whose image lies inside $B_{D'}(y)$. We now extend $g_s$ over $\sigma$ using this null homotopy.\\
To construct a homotopy from $g_s\circ f_s$ to $i\colon Y\hookrightarrow X$ one does an induction over the simplices of $Y\times I$ (best using a cell structure such that all cells are of the form $\sigma\times\{0\},\sigma\times\{1\}$ or $\sigma\times I$) using the relative uniform contractibility of $X$ with respect to $Y$. To do this boundedly one uses the uniform bound on the diameter of the cells of $Y$.
\end{proof}
\begin{rem}
\label{compare1}
If $X=\coprod_{i\in I} X_i$ is a metric CW complex with FDC, is uniformly contractible with respect to a subcomplex $Y=\coprod_{i\in I} Y_i$ and there exists a subspace with bounded geometry $S=\coprod_{i\in I} S_i$ such that $Y\subseteq S^R$ for some $R$, then for any family $\mcA_i$ of additive categories the above theorem yields maps
\[(f_s)_*\colon\prod_{i\in I}^{bd}(\mcA_i)_c(Y_i)\rightarrow \prod_{i\in I}^{bd}(\mcA_i)_c(P_s(S_i)),~(g_s)_*\colon\prod_{i\in I}^{bd}(\mcA_i)_c(P_s(S_i))\rightarrow \prod_{i\in I}^{bd}(\mcA_i)_c(X_i)\]
and $(g_s)^*\circ(f_s)^*$ induces a map on $K$-theory that is homotopic to the inclusion $\prod_{i\in I}^{bd}(\mcA_i)_c(Y_i)\rightarrow\prod_{i\in I}^{bd}(\mcA_i)_c(X_i)$ by \cref{rem:functorial}.
\end{rem}
\begin{thm}[{\cite[Theorem 6.4]{k-theory}}]
If $X=\coprod_{i\in I}X_i$ is a bounded geometry metric space with FDC, $\mcA$ an additive category, then for each $n\in\bbZ$ we have
\[\colim_s\pi_n\bbK\prod^{bd}_{i\in I}\mcA_c(P_sX_i)=\colim_s\pi_n\bbK(\mcA_c(P_sX))=0.\]
\end{thm}
The same proof as in \cite{k-theory} also yields a more general result.
\begin{thm}
\label{vanishing}
If $X=\coprod_{i\in I}X_i$ is a bounded geometry metric space with FDC, $\{A_i\}_{i\in I}$ a family of additive categories, then for each $n\in\bbZ$ we have
\[\colim_s\pi_n\bbK\prod^{bd}_{i\in I}(\mcA_i)_c(P_sX_i)=0.\]
\end{thm}
\section{The Descent Principle}
\label{des}
Let $Z$ be a simplicial $\Gamma$-CW complex and $\mcA$ a filtered, additive category with \mbox{$\Gamma$-action.} Let $J_k$ be the set of $k$-simplices in the barycentric subdivision of $Z$. Since the vertices of every simplex in the barycentric subdivision are naturally ordered by the inclusion of the corresponding simplices in $Z$, we get maps \[s_i\colon J_k\rightarrow J_{k-1},\sigma\mapsto \partial_i\sigma\quad\textnormal{for }0\leq i \leq k.\] Define for each $n\in\bbN$ \[A_k^n:=\Map_\Gamma(\Delta^k,(\bbK\prod_{J_k}^{bd}\mcA)_n)\cong\Map(\Delta^k,(\bbK\prod_{J_k}^{bd}\mcA)_n^\Gamma)\] and \[B_k^n:=\prod_{i=0}^k\Map_\Gamma(\Delta^{k-1},(\bbK\prod_{J_k}^{bd}\mcA)_n),\] where $(\bbK\prod_{J_k}^{bd}\mcA)_n$ is the n-th space of the spectrum $\bbK\prod_{J_k}^{bd}\mcA$. The maps $s_i$ induce maps $f_{k}^n:=(s_i^*)_i\colon A^n_{k-1}\rightarrow B_k^n$ and the inclusions $d_i\colon\Delta^{k-1}\rightarrow \Delta^k$ induce maps $g_k^n:=(d_i^*)_i\colon A_k^n\rightarrow B_k^n$.
\begin{defi}
\label{mappingspace}
The \emph{bounded mapping space} $\Map^{bd}_\Gamma(Z,\bbK\mcA)$ is defined as the spectrum whose $n$-th space is the subspace of $\prod_{k\in \bbN}A_k^n$ consisting of all $(h_k)_k\in\prod_{k\in\bbN}A_k^n$ with $f_{k}^n(h_{k-1})=g_k^n(h_k)$ for all $k\geq1$. The structure maps are induced by the structure maps of the spectra $\bbK(\prod_{J_k}^{bd}\mcA)$.
\end{defi}
\begin{rem}
\label{maps}
The inclusion $\prod_{J_k}^{bd}\mcA\hookrightarrow\prod_{J_k}\mcA$ together with the map $\bbK\big(\prod_{J_k}\mcA\big)_n\to \prod_{J_k}(\bbK\mcA)_n$ induces a map 
\[F_k\colon\Map_\Gamma\Bigg(\Delta^k,\bigg(\bbK\prod^{bd}_{J_k}\mcA\bigg)_n\Bigg)\rightarrow\Map_\Gamma\Bigg(\Delta^k,\prod_{J_k}(\bbK\mcA)_n\Bigg)\cong\Map_\Gamma \Bigg(\coprod_{J_k}\Delta^k,(\bbK\mcA)_n\Bigg).\]
For $\sigma\in J_k$ let $F_k(h_k)(\sigma)$ denote the restriction of $F_k(h_k)$ to the $\sigma$-component. Since $f_{k}^n(h_{k-1})=g_k^n(h_k)$ for every $(h_k)_k\in\left(\Map^{bd}_\Gamma(Z,\bbK\mcA)\right)_n$, we get
\[F_k(h_k)(\sigma)|_{\partial_i}=d_i^*F_k(h_k)(\sigma)=s_i^*F_{k-1}(h_{k-1})(\sigma)=F_{k-1}(h_{k-1})(\partial_i\sigma).\]
For $\sigma_1,\sigma_2\in J_k$ with $\partial_i\sigma_1=\partial_j\sigma_2$ this implies
\[F_k(h_k)(\sigma_1)|_{\partial_i}=F_{k-1}(h_{k-1})(\partial_i\sigma_1)=F_{k-1}(h_{k-1})(\partial_j\sigma_2)=F_k(h_k)(\sigma_2)|_{\partial_j}.\]
This shows that the maps 
\[F_k(h_k)\in \Map_\Gamma \bigg(\coprod_{J_k}\Delta^k,(\bbK\mcA)_n\bigg)\] fit together to a map $h\in\left(\Map_\Gamma(Z,\bbK\mcA)\right)_n$. Therefore, the inclusions \mbox{$\prod_{J_k}^{bd}\mcA\hookrightarrow\prod_{J_k}\mcA$} induce a map
\[F\colon\Map^{bd}_\Gamma(Z,\bbK\mcA)\rightarrow\Map_\Gamma(Z,\bbK\mcA).\]
Furthermore, the diagonal map $\Delta\colon\mcA\rightarrow\prod^{bd}_{J_k}\mcA$ induces a map 
\[G\colon\bbK(\mcA^\Gamma)\rightarrow\Map_\Gamma^{bd}(Z,\bbK\mcA)\]
by sending $x\in(\bbK\mcA^\Gamma)_n$ to $(h_k)_k\in\Map_\Gamma^{bd}(Z,\bbK\mcA)_n$ with $h_k\equiv \bbK(\Delta)(x)$ for all $k$. The composition $F\circ G\colon \bbK(\mcA^\Gamma)\cong\Map_\Gamma(*,\bbK\mcA)\to \Map_\Gamma(Z,\bbK\mcA)$ is induced by the map $Z\to*$.
\end{rem}
Next we will show that $\Map^{bd}_\Gamma(Z,\bbK\mcA)$ can be characterized as a homotopy limit. We will need this later on to see that it commutes with other homotopy limits.
\begin{prop}
\label{holim}
Let $Z$ be a simplicial $\Gamma$-CW complex and $\mcA$ a filtered, additive category with $\Gamma$-action. Fix $n\in\bbN$ and  let $(A_k^n,B_k^n,f_k^n,g_k^n)$ be as above. Then $\left(\Map_\Gamma^{bd}(Z,\bbK\mcA)\right)_n$ is a model for the limit as well as the homotopy limit of the diagram $(A_k^n,B_k^n,f_k^n,g_k^n)$.
\end{prop}
We will use that pullbacks where one of the two maps is a fibration are homotopy pullbacks and that the limit of a tower of fibrations is a homotopy limit of that tower. These facts are well known and the analogous statements in the category of simplicial sets instead of topological spaces can be found in \cite[Chapter XI, Examples 4.1(iv)\&(v)]{bk}.
\begin{proof}
Let $M_m\subseteq \prod_{k\leq m}A_k^n$ denote the subspace with $f_k^n(h_{k-1})=g_k^n(h_k)$. $M_m$ is a limit of the diagram $(A_k^n,B_k^n,f_k^n,g_k^n)_{k\leq m}$. The limit arises from taking finitely many pullbacks. Since the maps $g_k^n$ are fibrations, the space $M_m$ is a also a homotopy limit of this diagram and the induced maps $M_m\to M_{m-1}$ are fibrations as well. $\left(\Map_\Gamma^{bd}(Z,\bbK\mcA)\right)_n$ is a limit of the tower
\[\ldots\rightarrow M_m\rightarrow M_{m-1}\rightarrow\ldots \rightarrow M_1\rightarrow M_0=A_0,\]
and since all these arrows are fibrations, it is also a homotopy limit of the tower. Therefore, $\left(\Map_\Gamma^{bd}(Z,\bbK\mcA)\right)_n$ is a model for the limit and the homotopy limit of $(A_k^n,B_k^n,f_k^n,g_k^n)$.
\end{proof}
\begin{prop}
\label{bounded}
Let $Y$ be a finite dimensional, simplicial $\Gamma$-CW complex with finite stabilizers and let $X$ be a $\Gamma$-CW complex such that for every $\Gamma$-set $J$ with finite stabilizers and every $n\in\bbN$
\[\colim_{K\subseteq X}\pi_n\bbK\left(\prod^{bd}_{J}\mcA_\Gamma(\Gamma K))\right)^\Gamma=0,\]
where the colimit is taken over all finite subcomplexes $K\subseteq X$. Then also
\[\colim_{K\subseteq X}\pi_n(\Map_\Gamma^{bd}(Y,\bbK\mcA_\Gamma(\Gamma K)))=0,~~\forall n\in\bbN.\]
\end{prop}
\begin{proof}
Let $x_0\in S^n$ be the base point and let $K\subseteq X$ be a finite subcomplex. As above let $J_k$ be the set of $k$-simplices in the barycentric subdivision of $Y$ and let $s_i\colon J_k\rightarrow J_{k-1}$ be defined by $\sigma\mapsto \partial_i\sigma$.\\
An element in $\pi_n(\Map_\Gamma^{bd}(Y,\bbK\mcA_\Gamma(\Gamma K)))$ is represented by a system of maps
\[h_k\in\Map_*(S^n,\Map(\Delta^k,\bbK(\prod_{J_k}^{bd}\mcA_\Gamma(\Gamma K))^\Gamma))\cong\Map_*(S^n\wedge \Delta^k_+,\bbK(\prod_{J_k}^{bd}\mcA_\Gamma(\Gamma K))^\Gamma)\]
such that
\begin{itemize}
\item $h_k|_{S^n\wedge(\partial_i\Delta^k)_+}=(s_i)^*\circ h_{k-1}$.
\end{itemize}
We will produce a null homotopy $\{H_k\}$ by induction on $k$. Since $h_0$ represents an element in $\pi_n\bbK(\prod_{J_0}^{bd}\mcA_\Gamma(\Gamma K))^\Gamma$, there exists a finite subcomplex $K'\supseteq K$ such that $h_0$ is null homotopic in $\bbK(\prod_{J_0}^{bd}\mcA_\Gamma(\Gamma K'))^\Gamma$ by assumption. Every such null homotopy gives a map \[H_0\in\Map_*(S^n\wedge \Delta^1,\bbK(\prod_{J_0}^{bd}\mcA_\Gamma(\Gamma K'))^\Gamma)\] with 
\begin{itemize}
\item $H_0|_{S^n\wedge(\partial_1\Delta^1\cup\{1\})}=h_0$.
\end{itemize}
where $\{1\}\in\Delta^1$ is the base point.
Now assume we already have constructed maps \[H_j\in\Map_*(S^n\wedge\Delta^{j+1},\bbK(\prod_{J_j}^{bd}\mcA_\Gamma(\Gamma K'))^\Gamma)\] ($\{j+1\}\in\Delta^{j+1}$ being the base point) for $j< k$ such that
\begin{itemize}
\item $H_j|_{S^n\wedge \partial_i\Delta^{j+1}}=(s_i)^*\circ H_{j-1},~0\leq i\leq j$ and
\item $H_j|_{S^n\wedge(\partial_{j+1}\Delta^{j+1}\cup\{j+1\})}=h_j$.
\end{itemize}
These glue together to a map \[\tilde H_k\in\Map_*(S^n\wedge\partial\Delta^{k+1},\bbK(\prod_{J_k}^{bd}\mcA_\Gamma(\Gamma K'))^\Gamma)\] such that
\begin{itemize}
\item $\tilde H_k|_{S^n\wedge \partial_i\Delta^{k+1}}=(s_i)^*\circ H_{k-1},~0\leq i\leq k$ and
\item $\tilde H_k|_{S^n\wedge(\partial_{k+1}\Delta^{k+1}\cup\{k+1\})}=h_k$.
\end{itemize}
Since \[S^n\wedge \partial\Delta^{k+1}\cong S^{n+k}\] the element $\tilde H_k$ gives an element in $\Map_*(S^{n+k},\bbK(\prod_{J_k}^{bd}\mcA_\Gamma(\Gamma K'))^\Gamma)$. By assumption there exists a finite subcomplex $K''\supseteq K'$ such that $\tilde H_k$ is null homotopic in $\Map_*(S^{n+k},\bbK(\prod_{J_k}^{bd}\mcA_\Gamma(\Gamma K''))^\Gamma)$. Any such null homotopic can be used to extend $\tilde H_k$ to a map \[H_k\in \Map_*(S^n\wedge \Delta^{k+1},\bbK(\prod_{J_k}^{bd}\mcA_\Gamma(\Gamma K''))^\Gamma)\] with the properties stated above.\\
Since $Y$ was assumed to be finite dimensional, after finitely many steps we have constructed the required null homotopy $\{H_k\}$.
\end{proof}
\begin{lemma}
Let $X$ be a metric space. Then 
\[\prod_{I}^{bd}\mcA_\Gamma(X)_0\rightarrow \prod_I^{bd}\mcA_\Gamma(X)\rightarrow \prod_I\mcA_\Gamma(X)^\infty\]
is a Karoubi filtration.
\end{lemma}
\begin{proof}
To show that the sequence is Karoubi it suffices to show that for every morphism $f\colon M\rightarrow N$ in $\mcA_\Gamma(X)^\infty$ and every $R>0$ there exists a morphism $\phi'$ in $\mcA_\Gamma(X)$ that is $R$-bounded and represents $f$.\\
Let $\phi$ be a representative of $f$. For every $x\in X$ let $U_x:=B_{R/2}(x)\times [0,1]\subseteq X\times [0,1]$. Since $\phi$ is continuously controlled at 1, there exists a neighborhood $V_x\subseteq U_x$ of $(x,1)\in X\times[0,1]$ such that $\phi_{(\gamma',y),(\gamma,v)}=0$ for all $\gamma,\gamma'\in\Gamma,v\in V_x,y\notin U_x$. Define $V:=\bigcup_{x\in X}V_x$. Then $M|_{\Gamma\times X\times[0,1]\backslash \Gamma\times V}$ is an object in $\mcA_\Gamma(X)_0$ and therefore the morphism $\phi'\colon M\rightarrow N$ defined by
\[\phi'_{(\gamma',y),(\gamma,v)}=\left\{\begin{matrix}\phi_{(\gamma',y),(\gamma,v)}&v\in V\\0&\text{else}\end{matrix}\right.\]
also represents $f$.\\
$\phi'$ is $R$-bounded, since $\phi'_{(\gamma',y),(\gamma,v)}\neq 0$ implies $v\in V_x\subseteq B_{R/2}(x)\times[0,1]$ and\linebreak$y\in U_x\subseteq B_{R/2}(x)\times[0,1]$ for some $x\in X$. Therefore, $d(\pr_X(v),\pr_X(y))<R$, where $\pr_X\colon X\times[0,1]\rightarrow X$ is the projection.
\end{proof}
By formally defining $\prod_I^{bd}\mcA_\Gamma(X)^\infty:=\prod_I\mcA_\Gamma(X)^\infty$ the above proposition and \cref{holim} imply that we get the following homotopy fibration sequence:
\[\Map^{bd}_\Gamma(Z,\bbK(\mcA_\Gamma(X)_0))\rightarrow \Map^{bd}_\Gamma(Z,\bbK(\mcA_\Gamma(X)))\rightarrow \Map^{bd}_\Gamma(Z,\bbK(\mcA_\Gamma(X)^\infty))\]
This can be used to prove the following version of the Descent Principle:
\begin{thm}
\label{descent}
Let $\Gamma$ be a discrete group admitting a finite dimensional model for $\underbar E\Gamma$, and let $X$ be a $\Gamma$-CW complex such that for every $\Gamma$-set $J$ with finite stabilizers, every finite subcomplex $K\subseteq X$ and every $x\in\pi_n(\bbK(\prod_{j\in J}^{bd}\mcA_\Gamma(\Gamma K))^\Gamma)$ there exists a finite subcomplex $K'\subseteq X$ containing $K$ such that under the map induced by the inclusion
\[\prod_{j\in J}^{bd}\mcA_\Gamma(\Gamma K)\rightarrow \prod_{j\in J}^{bd}\mcA_\Gamma(\Gamma K'),\]
$x$ maps to zero. Then the map
\[H_*^\Gamma(X;\bbK_\mcA)\rightarrow K_*(\mcA[\Gamma])\]
is a split injection.
\end{thm}
\begin{proof}
Let $Y$ be a finite dimensional, simplicial $\Gamma$-CW-model for $\underbar E\Gamma$.\\
Consider the following commutative diagram:
\[\xymatrix{
\bbK(\mcA_\Gamma(\Gamma K)_0)^\Gamma\ar[d]\ar[r]&\bbK(\mcA_\Gamma(\Gamma K))^\Gamma\ar[r]\ar[d]&\bbK(\mcA_\Gamma(\Gamma K)^\infty)^\Gamma\ar[d]^f\\
\Map_\Gamma^{bd}(Y,\bbK(\mcA_\Gamma(\Gamma K)_0))\ar[r]\ar[d]&\Map_\Gamma^{bd}(Y,\bbK(\mcA_\Gamma(\Gamma K)))\ar[r]\ar[d]&\Map_\Gamma^{bd}(Y,\bbK(\mcA_\Gamma(\Gamma K)^\infty))\ar[d]^g\\
\Map_\Gamma(Y,\bbK(\mcA_\Gamma(\Gamma K)_0))\ar[r]&\Map_\Gamma(Y,\bbK(\mcA_\Gamma(\Gamma K)))\ar[r]&\Map_\Gamma(Y,\bbK(\mcA_\Gamma(\Gamma K)^\infty))}\]
All three rows in this diagram are induced by Karoubi filtrations and are, therefore, homotopy fibrations. The vertical maps are those from \cref{maps}. The composition $g\circ f$ is a weak homotopy equivalence by \cite[Theorem 6.3]{rosenthal}. Therefore, $f$ induces a split injection on homotopy groups.

Let $p\colon X\to \Gamma\backslash X$ be the projection map. Then we have the following equivalence
\[\mcA_{p,\Gamma}(X)\simeq \colim_{A\subseteq X}\mcA_\Gamma(A)\simeq\colim_{K\subseteq X}\mcA_\Gamma(\Gamma K),\]
where the first colimit is taken over all cocompact subsets $A\subseteq X$ and the second over all finite subcomplexes $K\subseteq X$.\\
%We have that $\colim_{K\subseteq X}\pi_n(\bbK(\mcA_\Gamma(\Gamma K)))\cong \pi_n\bbK(\mcA_{p,\Gamma}(X))$, where the colimit is taken over all finite subcomplexes $K\subseteq X$ and $p\colon X\rightarrow \Gamma\backslash X$ is the projection map. 
Thus, after taking homotopy groups and colimits over finite subcomplexes $K\subseteq X$ we get the following:
\[\xymatrix{
\pi_{n+1}\bbK(\mcA_{p,\Gamma}(X)^\infty)^\Gamma\ar[d]^{f_*}\ar[r]^\partial&\pi_n\bbK(\mcA_{p,\Gamma}(X)_0)\ar[d]\\
\colim_K\pi_{n+1}(\Map^{bd}_\Gamma(Y,\bbK(\mcA_\Gamma(\Gamma K)^\infty)))\ar[r]^\partial&\colim_K\pi_n(\Map_\Gamma^{bd}(Y,\bbK(\mcA_\Gamma(\Gamma K)_0)))}\]
The lower horizontal map is an isomorphism by \cref{bounded} and $f_*$ is split injective as stated above. So the upper horizontal map is split injective. This map is equivalent to the map in the theorem.
\end{proof}
\begin{rem}
By \cite{MR1351941} $K$-theory commutes with products. Since we defined $\prod^{bd}_I\mcA_\Gamma(X)^\infty=\prod_I\mcA_\Gamma(X)^\infty$ this implies that the maps
\[F_k\colon\Map_G\Bigg(\Delta^k,\bigg(\bbK\prod^{bd}_{J_k}\mcA_\Gamma(X)^\infty\bigg)_n\Bigg)\rightarrow\Map_G \Bigg(\coprod_{J_k}\Delta^k,(\bbK\mcA_\Gamma(X)^\infty)_n\Bigg).\]
from \cref{maps} are weak equivalences. Therefore, also the map
\[g\colon \Map_G^{bd}(Y,\bbK(\mcA_\Gamma(\Gamma K)^\infty))\to\Map_G(Y,\bbK(\mcA_\Gamma(\Gamma K)^\infty))\]
from the proof of \cref{descent} is a weak equivalence. Since $g\circ f$ is a weak equivalence by \cite[Theorem 6.3]{rosenthal} as mentioned also the map
\[f\colon \bbK(\mcA_\Gamma(\Gamma K)^\infty)^\Gamma\to \Map_G^{bd}(Y,\bbK(\mcA_\Gamma(\Gamma K)^\infty))\]
is a weak equivalence.
\end{rem}
\section{Proof of the main theorem}
\label{proof}
\begin{thm}
\label{main}
Let $\Gamma$ be a discrete group with fqFDC and with an upper bound on the order of its finite subgroups and let $\mcA$ be an additive $\Gamma$-category. Assume that there is a finite dimensional $\Gamma$-CW-model for the universal space for proper \mbox{$\Gamma$-actions $\underbar E\Gamma$}.\\
Then the assembly map in algebraic $K$-theory $H_*^\Gamma(\underbar E\Gamma;\bbK_\mcA)\rightarrow K_*(\mcA[\Gamma])$ is a split injection.
\end{thm}
Let $X$ be a finite dimensional, simplicial $\Gamma$-CW-model for $\underbar E\Gamma$ with a proper \mbox{$\Gamma$-invariant} metric. Such a model exists by the assumptions of \cref{main} and \cref{metric}.\\
Let $G$ be a finite subgroup of $\Gamma$. Let $G=H_0^G,H_1^G,\ldots,H_{m_G}^G=\{e\}$ be a representing system for the conjugacy classes of subgroups of $G$ ordered by cardinality, that is $|H_i^G|\geq |H_{i+1}^G|$.\\
Let $m:=\max_{G}m_G$ and define $H_l^G:=\{e\}$ for all $m_G\leq l\leq m$.\\
\\
For each $k,~0\leq k\leq m$, define $\mcS_k^G:=\{(H_i^G)^g\mid 0\leq i\leq k,g\in G\}$ and $Z_k^G:=X^{\mcS_k^G}$, see \cref{nota:1.9}. Since $\mcS_k^G$ is closed under conjugation by $G$, the space $Z_k^G$ is $G$-invariant for every $k$.
\begin{lemma}
\label{xx}
Under the assumptions of \cref{main}, for every $k$ ($0\leq k\leq m$), every $n\in\bbZ$ and every family of finite subgroups $\{G_i\}_{i\in I}$ of $\Gamma$, we have
\[\colim_K\pi_n\bbK\left(\prod_{i\in I}^{bd}\mcA[H_k^{G_i}]_c(G_i\backslash (Z_k^{G_i}\cap\Gamma K))\right)=0,\]
where the colimit is taken over all finite subcomplexes $K\subseteq X$.
\end{lemma}
\begin{proof}
Let $K\subseteq X$ be a finite subcomplex. By \cref{uni1} and \cref{uni2} the space $\coprod_{i\in I}G_i\backslash Z_k^{G_i}$ is uniformly contractible with respect to $\coprod_{i\in I}G_i\backslash(Z_k^{G_i}\cap \Gamma K)$.\\
Let $x\in K$ and choose $R>0$ such that $K\subseteq B_R(x)$. Since $Z_k^{G_i}\cap\Gamma K$ is $G_i$-invariant, $(Z_k^{G_i}\cap\Gamma K)^R$ is $G_i$-invariant as well. Choose maps \[\rho_i\colon G_i\backslash (\Gamma x\cap (Z_k^{G_i}\cap\Gamma K)^R)\rightarrow G_i\backslash (Z_k^{G_i}\cap\Gamma K)\] with $d(y,\rho_i(y))\leq R$ for all $y\in G_i\backslash (\Gamma x\cap (Z_k^{G_i}\cap\Gamma K)^R)$ and define $S_i:=\im(\rho_i)$.\\
We have $G_i\backslash (\Gamma K\cap Z_k^{G_i})\subseteq S_i^{2R}$, since $(Z_k^{G_i}\cap\Gamma K)\subseteq(\Gamma x\cap (Z_k^{G_i}\cap\Gamma K)^R)^R$. Furthermore, $\coprod_{i\in I}S_i\subseteq \coprod_{i\in I}G_i\backslash (Z_k^{G_i}\cap\Gamma K)$ is a subspace with bounded geometry because $\Gamma x$ has bounded geometry. So by \cref{compare} and \cref{compare1} for every $s>4R$ the inclusion \[\bbK\left(\prod_{i\in I}^{bd}\mcA[H_k^{G_i}]_c(G_i\backslash (Z_k^{G_i}\cap\Gamma K))\right)\rightarrow\bbK\left(\prod_{i\in I}^{bd}\mcA[H_k^{G_i}]_c(G_i\backslash Z_k^{G_i})\right)\] factors through $\bbK\left(\prod_{i\in I}^{bd}\mcA[H_k^{G_i}]_c(P_s(S_i))\right)$. Since the maps $f_s,g_s$ constructed in \cref{compare} are metrically coarse, there exits $R'>0$ such that \[g_s\circ f_s(G_i\backslash (Z_k^{G_i}\cap\Gamma K))\subseteq (G_i\backslash (Z_k^{G_i}\cap\Gamma K))^{R'}\] and since $X$ is proper there exists a finite subcomplex $K'$ containing $K^{R'}$. This shows that already the inclusion
\[\bbK\left(\prod_{i\in I}^{bd}\mcA[H_k^{G_i}]_c(G_i\backslash (Z_k^{G_i}\cap\Gamma K))\right)\rightarrow\bbK\left(\prod_{i\in I}^{bd}\mcA[H_k^{G_i}]_c(G_i\backslash (Z_k^{G_i}\cap\Gamma K'))\right)\] factors through $\bbK\left(\prod_{i\in I}^{bd}\mcA[H_k^{G_i}]_c(P_s(S_i))\right)$. \\
Define a metric on $\Gamma$ by 
\[d(\gamma,\gamma')=\left\{\begin{matrix}0&\gamma=\gamma'\\2+d(\gamma x,\gamma' x)&\gamma\neq\gamma'\end{matrix}\right..\]
Since the stabilizer of $x$ is finite, this metric is proper. The map $\Gamma\rightarrow\Gamma x,\gamma\mapsto \gamma x$ is a coarse equivalence and therefore $\Gamma x$ has fqFDC by \cref{coarse invariance}.\\
$\coprod_{i\in I}S_i$ is coarsely equivalent to a subspace of $\coprod_{i\in I}G_i\backslash\Gamma x$ and so $\coprod_{i\in I}S_i$ has FDC.\\
It follows that $\colim_s\pi_n\bbK\left(\prod_{i\in I}^{bd}\mcA[H_k^{G_i}]_c(P_s(S_i))\right)=0$ by \cref{vanishing}. Therefore, the colimit in the statement of the proposition also vanishes.
\end{proof}
The following lemma is essentially the same as \cite[8.4]{bartels}.
\begin{lemma}
\label{x}
For each $k$, $1\leq k\leq m$, every finite subgroup $G\leq \Gamma$ and every finite subcomplex $K\subseteq X$ we have the following equivalence:
\[\mcA^{G}_\Gamma(Z_k^G\cap \Gamma K,Z^G_{k-1}\cap \Gamma K)\simeq \mcA[H_k^G]_c(G\backslash (Z_k^G\cap\Gamma K),G\backslash(Z^G_{k-1}\cap\Gamma K)).\]
\end{lemma}
~\\
~\\
\textit{Proof of \cref{main}~~}\\
By the \hyperref[descent]{Descent Principle \ref*{descent}} it suffices to show that for every integer $n$ and every $\Gamma$-set~$J$ with finite stabilizers the following holds
\[\colim_K\pi_n\bbK\left(\prod_{j\in J}^{bd}\mcA_\Gamma(\Gamma K)\right)^\Gamma=0\]
where the colimit is taken over all finite subcomplexes $K\subseteq X$.\\
But since $\left(\prod_{j\in J}^{bd}\mcA_\Gamma(\Gamma K)\right)^\Gamma$ is equivalent to $\prod_{\Gamma j\in\Gamma\backslash J}^{bd}\mcA_\Gamma^{\Gamma_j}(\Gamma K)$, where $\Gamma_j$ is the stabilizer of $j\in J$, this is equivalent to showing that for every family of finite subgroups~$\{G_i\}_{i\in I}$ over some index set $I$ the following holds
\[\colim_K\pi_n\bbK\left(\prod_{i\in I}^{bd}\mcA^{G_i}_\Gamma(\Gamma K))\right)=0.\]
We will proceed by induction on the filtrations
\[\underbar E\Gamma^{G_i} = Z_0^{G_i}\subseteq Z_1^{G_i}\subseteq\ldots\subseteq Z_{m-1}^{G_i}\subseteq Z_m^{G_i} = \underbar E\Gamma\]
defined above.\\
Since $G_i$ acts trivially on $\underbar E\Gamma^{G_i}$, $\mcA^{G_i}_\Gamma(\underbar E\Gamma^{G_i}\cap\Gamma K)$ is equivalent to $\mcA[G_i]_c(\underbar E\Gamma^{G_i}\cap\Gamma K)$ and by \cref{xx} we have
\[\colim_K\pi_n\bbK\prod_{i\in I}^{bd}\mcA[G_i]_c(\underbar E\Gamma^{G_i}\cap\Gamma K)=0.\]
This completes the base case of the induction.\\
\\
Assume now \[\colim_K\pi_n\bbK\prod_{i\in I}^{bd}\mcA^{G_i}_\Gamma(Z_k^{G_i}\cap\Gamma K)=0.\] We must show that also \[\colim_K\pi_n\bbK\prod_{i\in I}^{bd}\mcA^{G_i}_\Gamma(Z_{k+1}^{G_i}\cap\Gamma K)=0.\]
Consider the following Karoubi filtration coming from \cref{prop:5.8}
\[\prod_{i\in I}^{bd}\mcA^{G_i}_\Gamma(Z_{k}^{G_i}\cap\Gamma K)\rightarrow \prod_{i\in I}^{bd}\mcA^{G_i}_\Gamma(Z_{k+1}^{G_i}\cap\Gamma K)\rightarrow \prod_{i\in I}^{bd}\mcA^{G_i}_\Gamma(Z_{k+1}^{G_i}\cap\Gamma K,Z_k^{G_i}\cap\Gamma K))\]
which yields a homotopy fibration of spectra after applying $\bbK$. By using the induction hypothesis, we only need to show that \[\colim_K\pi_n\bbK\prod_{i\in I}^{bd}\mcA^{G_i}_\Gamma(Z_{k+1}^{G_i}\cap\Gamma K,Z_k^{G_i}\cap\Gamma K)=0.\]
By \cref{x} \[\prod_{i\in I}^{bd}\mcA^{G_i}_\Gamma(Z_{k+1}^{G_i}\cap\Gamma K,Z_k^{G_i}\cap\Gamma K)\] is equivalent to \[\prod_{i\in I}^{bd}\mcA[H_k^{G_i}]_c(G_i\backslash(Z_{k+1}^{G_i}\cap\Gamma K),G_i\backslash(Z_k^{G_i}\cap\Gamma K)),\] which fits into the Karoubi sequence
\[\prod_{i\in I}^{bd}\mcA[H_k^{G_i}]_c(G_i\backslash(Z_k^{G_i}\cap\Gamma K))\rightarrow \prod_{i\in I}^{bd}\mcA[H_k^{G_i}]_c(G_i\backslash(Z_{k+1}^{G_i}\cap\Gamma K))\]
\[\rightarrow \prod_{i\in I}^{bd}\mcA[H_k^{G_i}]_c(G_i\backslash(Z_{k+1}^{G_i}\cap\Gamma K),G_i\backslash(Z_k^{G_i}\cap\Gamma K)).\]
By \cref{xx} we have
\[\colim_K\pi_n\bbK\prod_{i\in I}^{bd}\mcA[H_k^{G_i}]_c(G_i\backslash(Z_k^{G_i}\cap\Gamma K))=\colim_K\pi_n\bbK\prod_{i\in I}^{bd}\mcA[H_k^{G_i}]_c(G_i\backslash(Z_{k+1}^{G_i}\cap\Gamma K))=0.\]
Therefore, also \[\colim_K\pi_n\bbK\prod_{i\in I}^{bd}\mcA[H_k^{G_i}]_c(G_i\backslash(Z_{k+1}^{G_i}\cap\Gamma K),G_i\backslash(Z_k^{G_i}\cap\Gamma K))=0.\]\qed
\section{L-theory}
\label{ltheory}
As in \cite{bartels} we get the following $L$-theoretic version of \cref{main}.
\begin{thm}
Let $\Gamma$ be a discrete group with fqFDC and $\mcA$ an additive\linebreak \mbox{$\Gamma$-category} with involution. Assume that there is a finite dimensional $\Gamma$-CW-model for the universal space for proper $\Gamma$-actions $\underbar E\Gamma$ and that there is an upper bound on the order of finite subgroups of $\Gamma$. Assume further that for every finite subgroup $G\leq\Gamma$ there is an $i_0\in\bbN$ such that for $i\geq i_0$, $K_{-i}(\mcA[G])=0$, where $\mcA$ is considered only as an additive category.\\
Then the assembly map in $L$-theory $H_*^\Gamma(\underbar E\Gamma;\bbL^{\langle-\infty\rangle}_\mcA)\rightarrow L^{\langle-\infty\rangle}_*(\mcA[\Gamma])$ is a split  injection.
\end{thm}
\begin{proof}
Everything we did works for $L$-theory as it works for $K$-theory with exception of the \hyperref[descent]{Descent Principle \ref*{descent}}. Here the additional assumption about the vanishing of $K_{-i}(\mcA[G])$ for large $i$ is needed because only then the $L$-theoretic analogue of \cite[Theorem 6.3]{rosenthal} holds. For more details on this see \cite[Section 4]{MR1341817}.
\end{proof}
\bibliographystyle{amsalpha}
\bibliography{fqFDC}

\providecommand{\bysame}{\leavevmode\hbox to3em{\hrulefill}\thinspace}
\providecommand{\MR}{\relax\ifhmode\unskip\space\fi MR }
\providecommand{\MRhref}[2]{%
  \href{http://www.ams.org/mathscinet-getitem?mr=#1}{#2}
}
\providecommand{\href}[2]{#2}
\begin{thebibliography}{ACFP94}

\bibitem[ACFP94]{MR1277522}
Douglas~R. Anderson, Francis~X. Connolly, Steven~C. Ferry, and Erik~K.
  Pedersen, \emph{Algebraic {$K$}-theory with continuous control at infinity},
  J. Pure Appl. Algebra \textbf{94} (1994), no.~1, 25--47. \MR{1277522
  (95b:19003)}

\bibitem[AS81]{alperin}
Roger~C. Alperin and Peter~B. Shalen, \emph{Linear groups of finite
  cohomological dimension}, Bull. Amer. Math. Soc. (N.S.) \textbf{4} (1981),
  no.~3, 339--341. \MR{609046 (82c:20087)}

\bibitem[Bar03]{squeezing}
Arthur~C. Bartels, \emph{Squeezing and higher algebraic {$K$}-theory},
  $K$-Theory \textbf{28} (2003), no.~1, 19--37. \MR{1988817 (2004f:19006)}

\bibitem[BK72]{bk}
A.~K. Bousfield and D.~M. Kan, \emph{Homotopy limits, completions and
  localizations}, Lecture Notes in Mathematics, Vol. 304, Springer-Verlag,
  Berlin, 1972. \MR{0365573 (51 \#1825)}

\bibitem[BR07a]{MR2294225}
Arthur Bartels and Holger Reich, \emph{Coefficients for the {F}arrell-{J}ones
  conjecture}, Adv. Math. \textbf{209} (2007), no.~1, 337--362. \MR{2294225
  (2008a:19002)}

\bibitem[BR07b]{bartels}
Arthur Bartels and David Rosenthal, \emph{On the {$K$}-theory of groups with
  finite asymptotic dimension}, J. Reine Angew. Math. \textbf{612} (2007),
  35--57. \MR{2364073 (2009a:19004)}

\bibitem[Car95]{MR1351941}
Gunnar Carlsson, \emph{On the algebraic {$K$}-theory of infinite product
  categories}, $K$-Theory \textbf{9} (1995), no.~4, 305--322. \MR{1351941
  (96m:19008)}

\bibitem[CH69]{MR0242151}
E.~H. Connell and John Hollingsworth, \emph{Geometric groups and {W}hitehead
  torsion}, Trans. Amer. Math. Soc. \textbf{140} (1969), 161--181. \MR{0242151
  (39 \#3485)}

\bibitem[CP95]{MR1341817}
Gunnar Carlsson and Erik~Kj{\ae}r Pedersen, \emph{Controlled algebra and the
  {N}ovikov conjectures for {$K$}- and {$L$}-theory}, Topology \textbf{34}
  (1995), no.~3, 731--758. \MR{1341817 (96f:19006)}

\bibitem[DL98]{davislueck}
James~F. Davis and Wolfgang L{\"u}ck, \emph{Spaces over a category and assembly
  maps in isomorphism conjectures in {$K$}- and {$L$}-theory}, $K$-Theory
  \textbf{15} (1998), no.~3, 201--252. \MR{1659969 (99m:55004)}

\bibitem[DP]{poschar}
Dieter Degrijse and Nansen Petrosyan, \emph{{Bredon cohomological dimensions
  for groups acting on {CAT(0)}-spaces}}, arXiv:1208.3884, to appear in Groups,
  Geometry, and Dynamics.

\bibitem[Gro93]{gromov}
M.~Gromov, \emph{Asymptotic invariants of infinite groups}, Geometric group
  theory, {V}ol.\ 2 ({S}ussex, 1991), London Math. Soc. Lecture Note Ser., vol.
  182, Cambridge Univ. Press, Cambridge, 1993, pp.~1--295. \MR{1253544
  (95m:20041)}

\bibitem[GTY12]{rigidity}
Erik Guentner, Romain Tessera, and Guoliang Yu, \emph{A notion of geometric
  complexity and its application to topological rigidity}, Invent. Math.
  \textbf{189} (2012), no.~2, 315--357. \MR{2947546}

\bibitem[GTY13]{fdc}
\bysame, \emph{Discrete groups with finite decomposition complexity}, Groups
  Geom. Dyn. \textbf{7} (2013), no.~2, 377--402. \MR{3054574}

\bibitem[Hat02]{hatcher}
Allen Hatcher, \emph{Algebraic topology}, Cambridge University Press,
  Cambridge, 2002. \MR{1867354 (2002k:55001)}

\bibitem[Kar70]{karoubi}
Max Karoubi, \emph{Foncteurs d\'eriv\'es et {$K$}-th\'eorie}, S\'eminaire
  {H}eidelberg-{S}aarbr\"ucken-{S}trasbourg sur la {K}th\'eorie (1967/68),
  Lecture Notes in Mathematics, Vol. 136, Springer, Berlin, 1970, pp.~107--186.
  \MR{0265435 (42 \#344)}

\bibitem[L{\"u}c00]{MR1757730}
Wolfgang L{\"u}ck, \emph{The type of the classifying space for a family of
  subgroups}, J. Pure Appl. Algebra \textbf{149} (2000), no.~2, 177--203.
  \MR{1757730 (2001i:55018)}

\bibitem[Mis01]{mislin}
Guido Mislin, \emph{On the classifying space for proper actions}, Cohomological
  methods in homotopy theory ({B}ellaterra, 1998), Progr. Math., vol. 196,
  Birkh\"auser, Basel, 2001, pp.~263--269. \MR{1851258 (2002f:55032)}

\bibitem[PW85]{MR802790}
Erik~K. Pedersen and Charles~A. Weibel, \emph{A nonconnective delooping of
  algebraic {$K$}-theory}, Algebraic and geometric topology ({N}ew {B}runswick,
  {N}.{J}., 1983), Lecture Notes in Math., vol. 1126, Springer, Berlin, 1985,
  pp.~166--181. \MR{802790 (87b:18012)}

\bibitem[Qui79]{MR549490}
Frank Quinn, \emph{Ends of maps. {I}}, Ann. of Math. (2) \textbf{110} (1979),
  no.~2, 275--331. \MR{549490 (82k:57009)}

\bibitem[Qui82]{MR669423}
\bysame, \emph{Ends of maps. {II}}, Invent. Math. \textbf{68} (1982), no.~3,
  353--424. \MR{669423 (84j:57011)}

\bibitem[Roe03]{roe}
John Roe, \emph{Lectures on coarse geometry}, University Lecture Series,
  vol.~31, American Mathematical Society, Providence, RI, 2003. \MR{2007488
  (2004g:53050)}

\bibitem[Ros04]{rosenthal}
David Rosenthal, \emph{Splitting with continuous control in algebraic
  {$K$}-theory}, $K$-Theory \textbf{32} (2004), no.~2, 139--166. \MR{2083578
  (2005g:19003)}

\bibitem[RTY14]{k-theory}
Daniel~A. Ramras, Romain Tessera, and Guoliang Yu, \emph{Finite decomposition
  complexity and the integral {N}ovikov conjecture for higher algebraic
  {$K$}-theory}, J. Reine Angew. Math. \textbf{694} (2014), 129--178.
  \MR{3259041}

\bibitem[Sch04]{MR2079996}
Marco Schlichting, \emph{Delooping the {$K$}-theory of exact categories},
  Topology \textbf{43} (2004), no.~5, 1089--1103. \MR{2079996 (2005k:18023)}

\bibitem[Sel60]{selberg}
Atle Selberg, \emph{On discontinuous groups in higher-dimensional symmetric
  spaces}, Contributions to function theory (internat. {C}olloq. {F}unction
  {T}heory, {B}ombay, 1960), Tata Institute of Fundamental Research, Bombay,
  1960, pp.~147--164. \MR{0130324 (24 \#A188)}

\bibitem[tD87]{tomdieck}
Tammo tom Dieck, \emph{Transformation groups}, de Gruyter Studies in
  Mathematics, vol.~8, Walter de Gruyter \& Co., Berlin, 1987. \MR{889050
  (89c:57048)}

\end{thebibliography}
\end{document}